\documentclass[final]{amsart}

\usepackage{amsmath,amssymb,mathrsfs}
\usepackage{dsfont}
\usepackage{upgreek}
\usepackage{color}

\usepackage[pdftex]{graphicx}
\usepackage{tikz,pgfplots}
\renewcommand{\sigma}{{\mathrm{\upsigma}}}
\newcommand{\V}{\mathscr{V}}

\renewcommand{\vec}[1]{{\mathchoice
                     {\mbox{\boldmath$\displaystyle{#1}$}}
                     {\mbox{\boldmath$\textstyle{#1}$}}
                     {\mbox{\boldmath$\scriptstyle{#1}$}}
                     {\mbox{\boldmath$\scriptscriptstyle{#1}$}}}}
\newcommand{\ip}[1]{\langle #1 \rangle} 

\renewcommand{\div}{\mathop\mathrm{div}\nolimits}

\renewcommand{\div}{\mathop\mathrm{div}\nolimits}
\newcommand{\curl}{\mathop\mathrm{curl}\nolimits}

\newcommand{\Ev}{\mathbb{E}}
\newcommand{\ave}[1]{\Ev\left\{ {#1} \right\}}
\newcommand{\Lop}{\mathcal{L}}
\renewcommand{\L}{\vec{L}}

\newcommand{\R}{\mathbb{R}}

\newcommand{\LPot}[1]{\L^2_{\mathrm{pot}}(\Omega,#1)}
\newcommand{\Lpot}{\LPot{T}}
\newcommand{\LSol}[1]{\L^2_{\mathrm{sol}}(\Omega,#1)}
\newcommand{\Lsol}{\LSol{T}}
\newcommand{\VPot}[1]{\V^2_{\text{pot}}(\Omega,#1)}
\newcommand{\Vpot}{\VPot{T}}
\newcommand{\VSol}[1]{\V^2_{\text{sol}}(\Omega,#1)}
\newcommand{\Vsol}{\VSol{T}}
\newcommand{\A}{\mathcal{A}}
\newcommand{\grad}[1]{\nabla {#1}}

\newcommand{\B}{\mathfrak{B}}

\newcommand{\F}{\mathcal{F}}
\newcommand{\J}{\mathcal{J}}
\renewcommand{\O}{\mathcal{D}}
\newcommand{\Es}{\mathscr{E}}

\newcommand{\Char}[1]{{{\displaystyle \mathds{1}}_{{}_{#1}}}}

\newcommand{\weak}{\stackrel{w}{\rightharpoonup}}
\newcommand{\weaks}{\stackrel{\scriptstyle{w^{\star}}}{\rightharpoonup}}

\newcommand{\eps}{\varepsilon}

\newcommand{\Norm}[2]{\left\|{#1}\right\|_{#2}}

\newcommand{\Rnn}{\R^{n\times n}}
\newcommand{\Rnnsym}{\Rnn_\mathrm{sym}}

\usepackage{amsthm}

\newtheorem{theorem}{Theorem} 
\newtheorem{lemma}{Lemma} 
\newtheorem{remark}{Remark} 
\newtheorem{definition}{Definition}
\newtheorem{corollary}{Corollary}
\newtheorem{proposition}{Proposition}

\newcommand{\abar}{a^{\scriptscriptstyle{0}}}
\newcommand{\ubar}{{u^{\scriptscriptstyle{0}}}}

\newcommand{\sbar}{\sigma^0}
\newcommand{\smallT}{\scriptscriptstyle{T}}

\newcommand{\tor}{\mathbb{T}}

\title[Homogenization of stationary and ergodic random media]{A Primer on Homogenization of
Elliptic PDEs with Stationary and Ergodic Random Coefficient Functions}

\author{Alen Alexanderian}
\address{Institute for Computational Engineering and Sciences, The University of Texas at Austin}
\email{alen@ices.utexas.edu}


\date{\today}

\usepackage{listings}
\numberwithin{equation}{section}  %
\begin{document}

\begin{abstract}
We study the problem of characterizing the effective (homogenized) properties of
materials whose diffusive properties are modeled with random fields.
Focusing on elliptic PDEs with stationary and ergodic random coefficient functions,
we provide a gentle introduction to the mathematical theory of homogenization of random media.
We also present numerical examples to elucidate the theoretical concepts and results.
\end{abstract}

\maketitle

\newcommand{\azero}{a^0}
\section{Introduction}

Homogenization is a branch of the theory of partial differential equations (PDEs)
which provides the mathematical basis for describing effective physical properties of
materials with inhomogeneous microstructures.  In this article, we study
homogenization of random media, i.e., materials whose physical properties are
modeled with random functions.  Major theoretical results on homogenization of
random media were developed first by G.C.~Papanicolaou and
S.R.S.\ Varadhan in~\cite{Papa79}, and S.\ Kozlov in~\cite{Kozlov79}.  The theory of
homogenization of random media (stochastic homogenization), in addition to the
usual analysis and PDE theory tools, relies on results from probability and
ergodic theory.  This intermixing of analysis and PDE theory concepts with
those of probability often makes this otherwise elegant theory difficult to
penetrate for those with a more PDE oriented background and who are less
familiar with the probabilistic concepts encountered in stochastic
homogenization. 

This article aims to provide a gentle introduction to stochastic homogenization
by focusing on a few key results and proving them in detail.  We consider
linear elliptic PDEs with stationary and ergodic coefficient functions, and
provide proofs of homogenization result in one space dimension and in several
space dimensions.  A summary of the requisite background materials is provided
with an expanded discussion of concepts from ergodic theory. The first
homogenization result we study concerns one-dimensional elliptic equations with
random coefficients. The proof of the one-dimensional result, which is
considerably simpler than the general $n$-dimensional case,  provides a first
exposure to combining probabilistic and functional analytic tools to derive
homogenization results.  Our discussion of the homogenization theorem in the
general $n$-dimensional case follows in similar lines as the
arguments given in~\cite{Kozlov94} with many details added to keep the concepts
and arguments accessible.  Moreover, to make the presentation
beginner-friendly, throughout the article we provide a number of motivating
numerical examples to illustrate the theoretical concepts and results that
follow.  

The target audience of this article includes graduate
students who are entering this field of research as well as mathematicians who are
new to stochastic homogenization.  The background assumed in the following is a
working knowledge of basic concepts in PDE theory, a course in linear
functional analysis, and basic concepts from measure-theoretic probability.  Reading
this article should aid those new to the field in transitioning to advanced
texts such as~\cite{Kozlov94,Chechkin} that provide a complete coverage of
stochastic homogenization.  One should also keep in mind that the general
theory of homogenization is not limited to the cases of periodic or stationary
and ergodic media, and can be applied to physical processes other
than diffusion. We refer the reader to the book~\cite{Tartar09} by L.\ Tartar,
where the author provides an in-depth presentation of mathematical theory of
homogenization as well as the historical background on development of
homogenization theory.

Let us begin our discussion of homogenization with an example.  In
Figure~\ref{fig:micro1}, we depict what a realization of a medium with random
microstructure might look like.  
\begin{figure}[ht]\centering
\includegraphics[width=0.5\textwidth]{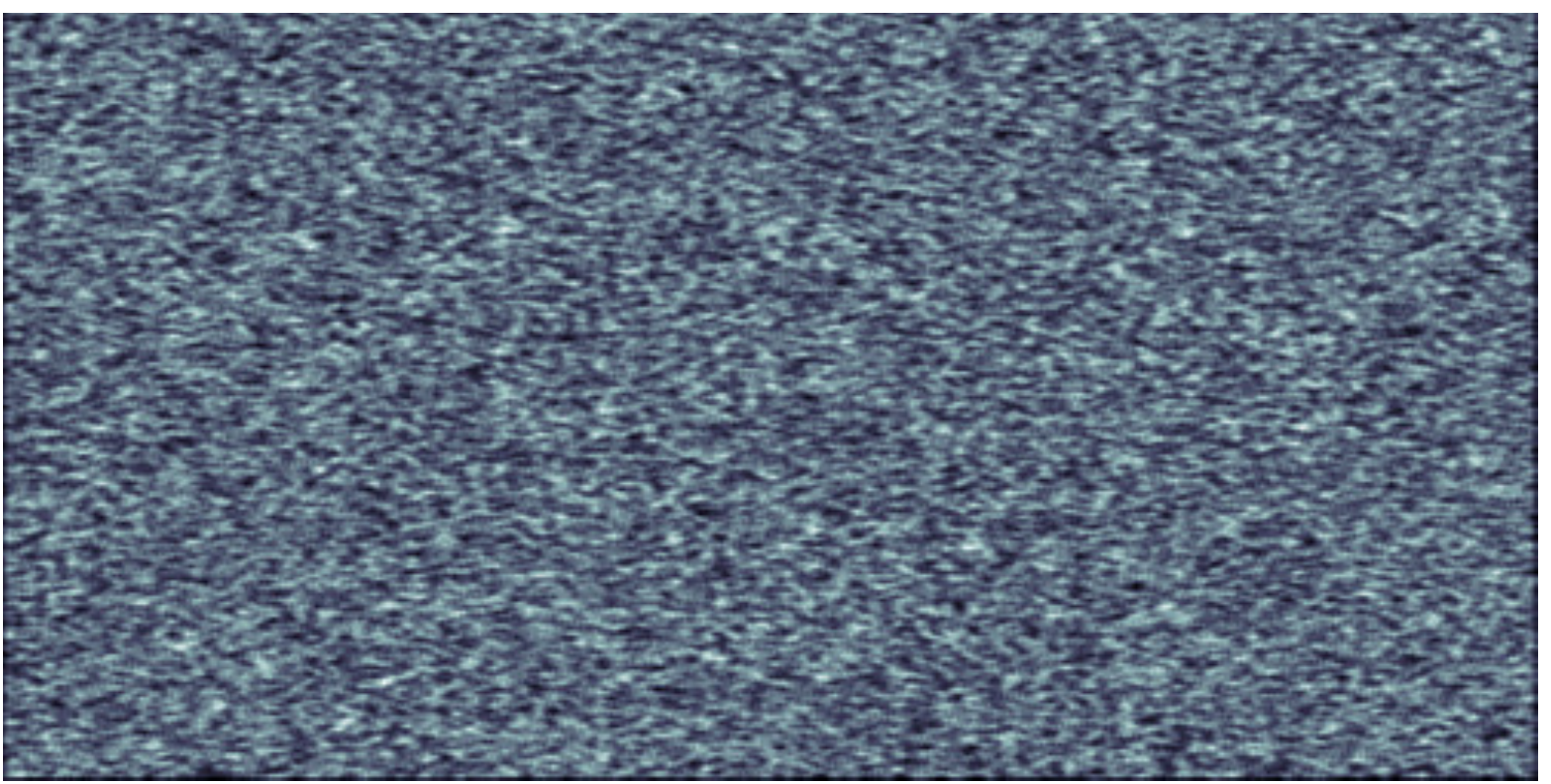} 
\caption{Depiction of a medium with random microstructure.} \label{fig:micro1} 
\end{figure}
Numerical modeling of physical processes such as diffusion through such media
is generally a challenging task, because the corresponding differential
equations have random coefficients whose realizations are rapidly oscillating
functions.  Given a diffusive medium with inhomogeneous (random)
microstructure, the goal of homogenization is to construct an effective
(homogenized) medium whose conductive/diffusive properties, in macroscale, are
close to the original medium.  The basic motivation for this is the fact that
the homogenized medium is much easier to work with.

To state the problem mathematically, we first consider a deterministic case. 
Let $A:\R^n \to \R^{n \times n}$  be a matrix valued coefficient function that is uniformly bounded 
and positive definite.  
We focus on elliptic differential operators of the form
\begin{equation}\label{equ:elliptic}
   \Lop^\eps u = -\div( A^\eps \grad u), \quad \mbox{ where } \quad A^\eps(\vec{x}) = A(\eps^{-1}\vec{x}),
\end{equation}
where $\vec{x} \in \R^n$ and $\eps > 0$ indicates a microstructural length-scale.
The coefficient functions $A^\eps$ characterize media with inhomogeneous microstructure. 
Homogenization theory studies the problem in the limit as $\eps \to 0$.

In the case of materials with random microstructure, the coefficient function $A$ in~\eqref{equ:elliptic} is a random 
field; i.e., $A = A(\vec{x}, \omega)$ where $\omega$ is an element of a sample space $\Omega$.
To motivate the basic questions that arise in homogenization, we consider some
specific numerical examples in Section~\ref{sec:motiv} below, in the context of
a problem in one space dimension.  This discussion is then used to guide the
reader through the subsequent sections of this article.

\newcommand{\ddx}{\frac{d}{dx}}
\newcommand{\dudx}{\frac{du}{dx}}
\newcommand{\dudxeps}{\frac{du^\eps}{dx}}
\section{Motivation and overview}\label{sec:motiv}
Although our discussion concerns mainly that of random structures, to develop 
some intuition we consider the case of a one-dimensional \emph{periodic} structure first. 
Consider the problem of modeling steady-state heat diffusion in a rod whose conductivity
profile is given by the function $a^\eps(x) = a(\eps^{-1} x)$ where $a$ is a bounded 
periodic function defined on the physical domain $\O$; in our example we let  $ \O = (0, 1)$.
Moreover, we assume that the temperature is fixed at zero at the end points of the interval. 
In this case, the following equation describes the steady-state temperature 
profile in the conductor,
\begin{equation} \label{equ:basic-prob-1D-per}
   \begin{aligned}
      -\ddx\left(a^\eps \dudxeps\right) = f \quad &\mbox{ in }  \O = (0,1),  \\
      u^\eps = 0 \quad &\mbox{ on } \partial\O = \{0, 1\}.
 \end{aligned}
\end{equation}
The right-hand side function $f$ describes a source term. Since $a$ is a periodic function, 
considering $a^\eps$ with successively smaller values of $\eps$ implies 
working with rapidly 
oscillating conductivity functions. 
Speaking in terms of material properties, considering successively smaller values of $\eps$
entails considering conductors with successively \emph{finer microstructure}. 
The basic question of homogenization is that of what happens as $\eps \to 0$, and whether
there is a limiting \emph{homogenized} material. 

For the purpose
of illustration, let us consider a specific example. We let the function $a(x)$ and the
right-hand side function $f(x)$ be given by
\begin{equation}\label{equ:1dprob_data} 
   a(x) = 2 + \sin(2\pi x), \quad
   f(x) = -3(2x - 1).
\end{equation}
It is clear that as $\eps \to 0$, the function $a^\eps$ becomes more and more oscillatory. 
In Figure~\ref{fig:sol_per} we plot the solution of the problem~\eqref{equ:basic-prob-1D-per} for the coefficient
functions $a^\eps$ with successively smaller values of $\eps$. 
\begin{figure}
\begin{tabular}{ccc}
\includegraphics[width=.3\textwidth]{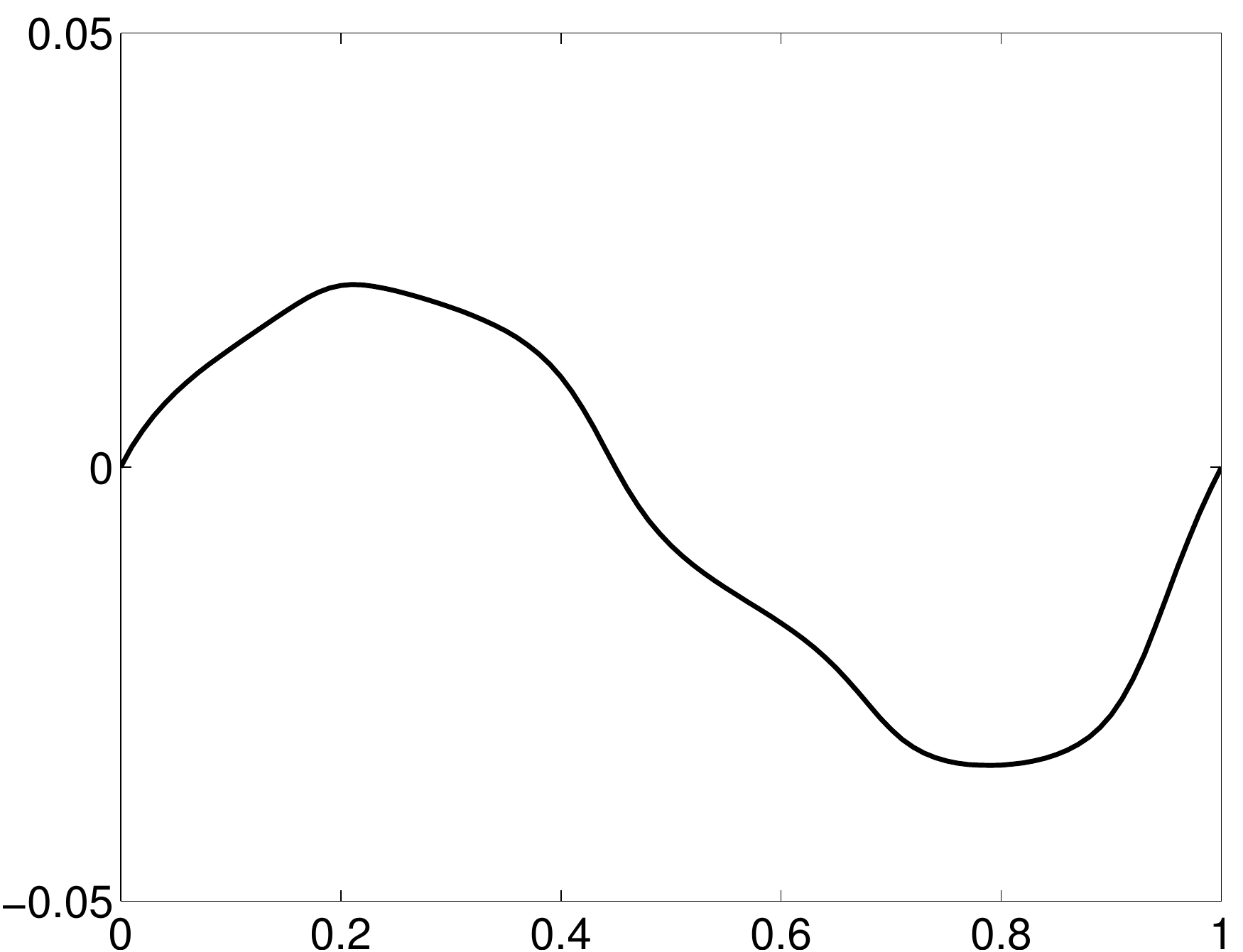}&
\includegraphics[width=.3\textwidth]{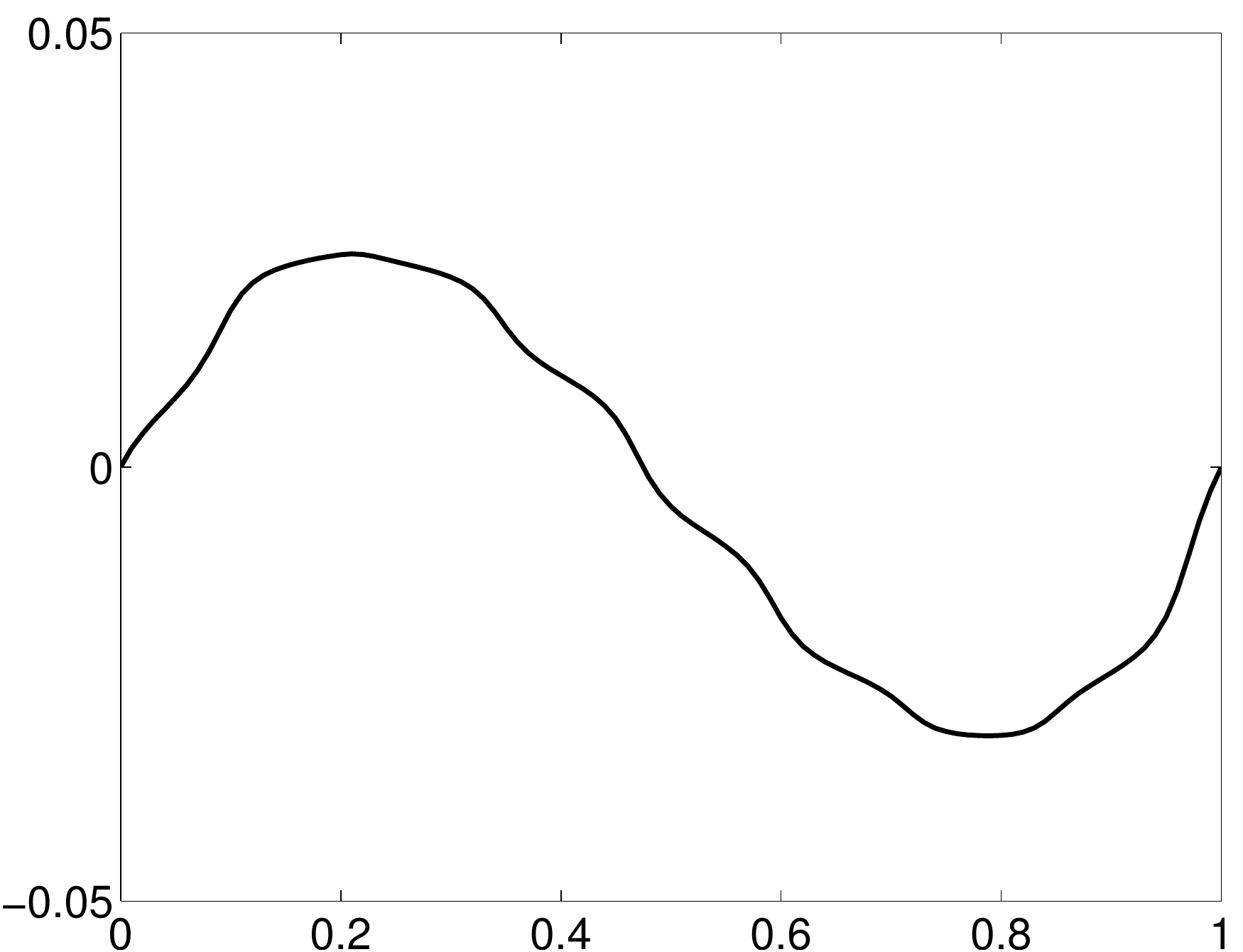}&
\includegraphics[width=.3\textwidth]{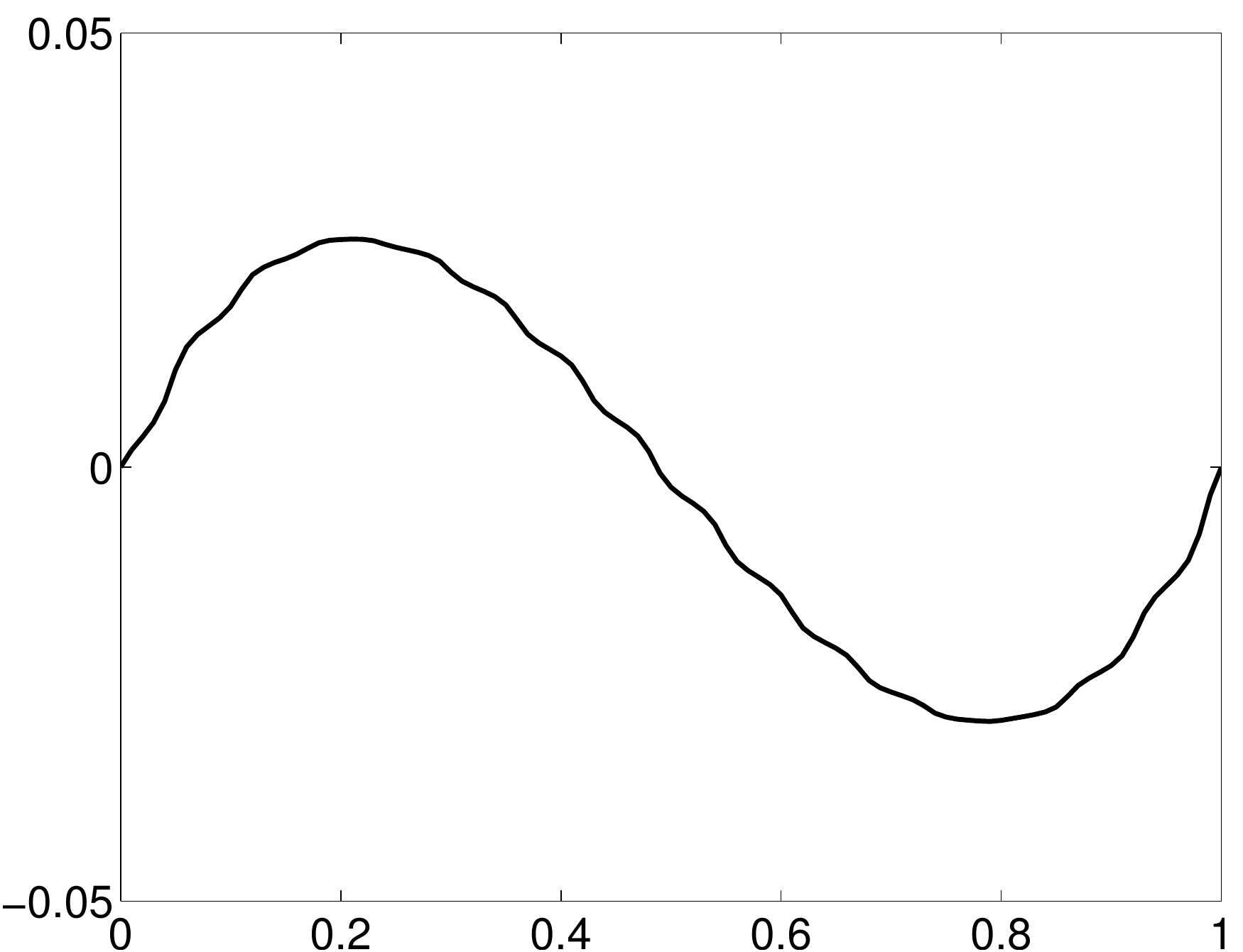}
\end{tabular}
\caption{The solutions $u^\eps$ corresponding to coefficient $a^\eps$ with $\eps = 1/4, 1/8$, and $1/16$ respectively.}
\label{fig:sol_per}
\end{figure}
The results plotted in Figure~\ref{fig:sol_per} suggest that
as $\eps$ gets smaller, the solutions $u^\eps$ seem to converge to a limit. The following are some relevant questions:
(1) Do $u^\eps$ actually converge to a limit? (2) If so in what topology does the convergence take place? (3) 
Can we describe/characterize the limit? The answers to these questions are all well-known.
In this case, the functions $u^\eps$ converge in $L^2(\O)$-norm to $\ubar$ that is the 
solution of the following problem:
\begin{equation}\label{equ:homog_equation_1d}
 \begin{aligned}
      -\ddx\left(\abar \frac{d\ubar}{dx} \right) = f \quad &\mbox{ in } \O = (0, 1),  \\
      \ubar = 0 \quad &\mbox{ on } \partial\O = \{0, 1\},
 \end{aligned}
\end{equation}
where $\abar$ is the harmonic mean of $a$ over the interval $(0, 1)$, 
\[
   \abar = \left( \int_0^1 \frac{1}{a(x)} \, dx \right)^{-1}. 
\]
The coefficient $\abar$ is called the \emph{homogenized coefficient} or the \emph{effective conductivity}.
Virtually every homogenization textbook or lecture note has some form of proof for this homogenization result. Hence, 
we just illustrate this result numerically here. Notice that with our choice of $a$ above, we 
have,
\[
\left( \int_0^1 \frac{1}{a(x)} \, dx \right)^{-1} = 
\left( \int_0^1 \frac{1}{2 + \sin(2\pi x)} \, dx \right)^{-1} =
\sqrt 3,
\]
as the homogenized coefficient. 
With this value of $\abar$, the analytic solution of the homogenized equation~\eqref{equ:homog_equation_1d} 
is given by,
\[
   \ubar(x) = \frac{1}{\sqrt{3}} x (x - 1/2) (x - 1). 
\]
In Figure~\ref{fig:convergence_per} we plot the function $\ubar$ (left plot) and 
demonstrate the convergence of $u^\eps$ to $\ubar$ by looking at $\| u^\eps - \ubar\|_{L^2(\O)}$ as $\eps \to 0$ (right plot). 
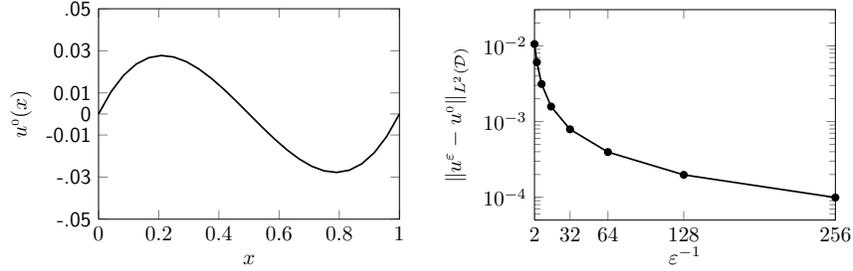
\begin{figure}[ht]\centering
\begin{tabular}{cc}
\begin{tikzpicture}[scale=.8]
\begin{axis}[
compat=newest,
width=5cm,height=3.5cm,domain=0:1,xmin=0,xmax=1,ymin=-.05,ymax=0.05,scale only axis,xlabel=$x$,
ylabel=$\ubar(x)$,
ytick={-.05,-.03,-0.01, 0, 0.01, 0.03,.05},
yticklabels={-.05,-.03,-0.01,0,0.01, 0.03,.05},
scaled y ticks = false
]
\addplot[thick]{1/sqrt(3)*x*(x-1/2)*(x-1)};
\end{axis}
\end{tikzpicture}&
\begin{tikzpicture}[scale=.8]
\begin{semilogyaxis}[width=5cm, height=3.5cm, scale only axis,ymin= 5e-5,ymax=3e-2,
    xmin = 2, xmax = 256,
    xtick={2,32,64,128,256},
    xlabel = ${\eps^{-1}}$, ylabel=$\|u^\eps - \ubar\|_{L^2(\O)}$,
compat=newest,
xlabel shift={-.15cm}
] 
\addplot [color=black, thick,mark=*,mark size=1.5pt] table[x=epsinv,y=err] {convplot_data.dat};
\end{semilogyaxis}
\end{tikzpicture}
\end{tabular}
\caption{Left: Plot of the solution $\ubar$ of the homogenized equation. 
Right: The convergence of the solutions $u^\eps$ to $\ubar$ in $L^2(\O)$; 
the black dots correspond to $\|u^{\eps_k} - \ubar\|_{L^2(\O)}$ with $\eps_k = 1/2^k$, $k = 1, \ldots, 8$.}
\label{fig:convergence_per}
\end{figure}

Now let us transition to the case of random media. In this case, the function $a$, 
which defines the conductivity profile of the material, is a random function. 
The stochastic version of~\eqref{equ:basic-prob-1D-per} is given by
\begin{equation} \label{equ:basic-prob-1D-rand}
   \begin{aligned}
      -\ddx \left(a^\eps(\cdot, \omega) \dudxeps(\cdot, \omega) \right) = f \quad &\mbox{ in }
      \O = (0, 1),  \\
      u^\eps(\cdot, \omega) = 0 \quad &\mbox{ on } \partial\O = \{0, 1\},
   \end{aligned}
\end{equation}
with $a^\eps(x, \omega) = a(\eps^{-1}x, \omega)$, and $a(x, \omega)$ a random function 
(random field). The variable $\omega$ is an element of a sample space $\Omega$, and for a fixed $\omega$, 
$a(\cdot, \omega)$ is a realization of the random function $a$. 
As an example, we consider a material
made up of tiles, each of which has conductivity of either $\kappa_1$ or $\kappa_2$, 
chosen randomly with probabilities $p$ and $1-p$ respectively, with $p \in (0, 1)$. 
A realization of the conductivity function for such a structure is depicted in Figure~\ref{fig:checker1d}, with the choices
of $\kappa_1 = 1$ and $\kappa_2 = 3$ and with $p = 1/2$. 
\begin{figure}[ht]\centering
\includegraphics[width=.45\textwidth]{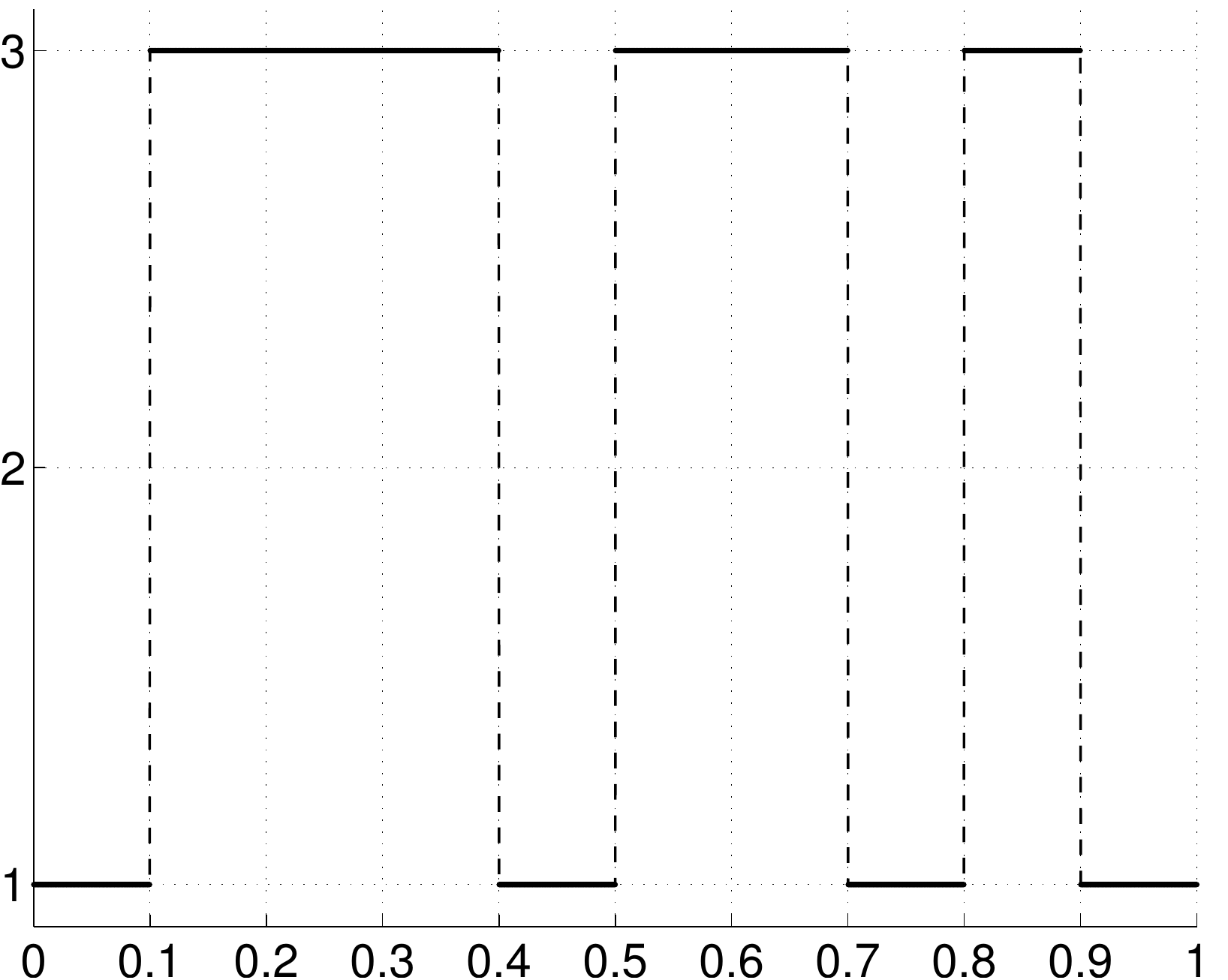}
\caption{A realization of the conductivity profile for a one-dimensional random checkerboard 
structure.}
\label{fig:checker1d}
\end{figure}
\begin{figure}[ht]\centering
\begin{tabular}{ccc}
\includegraphics[width=.3\textwidth]{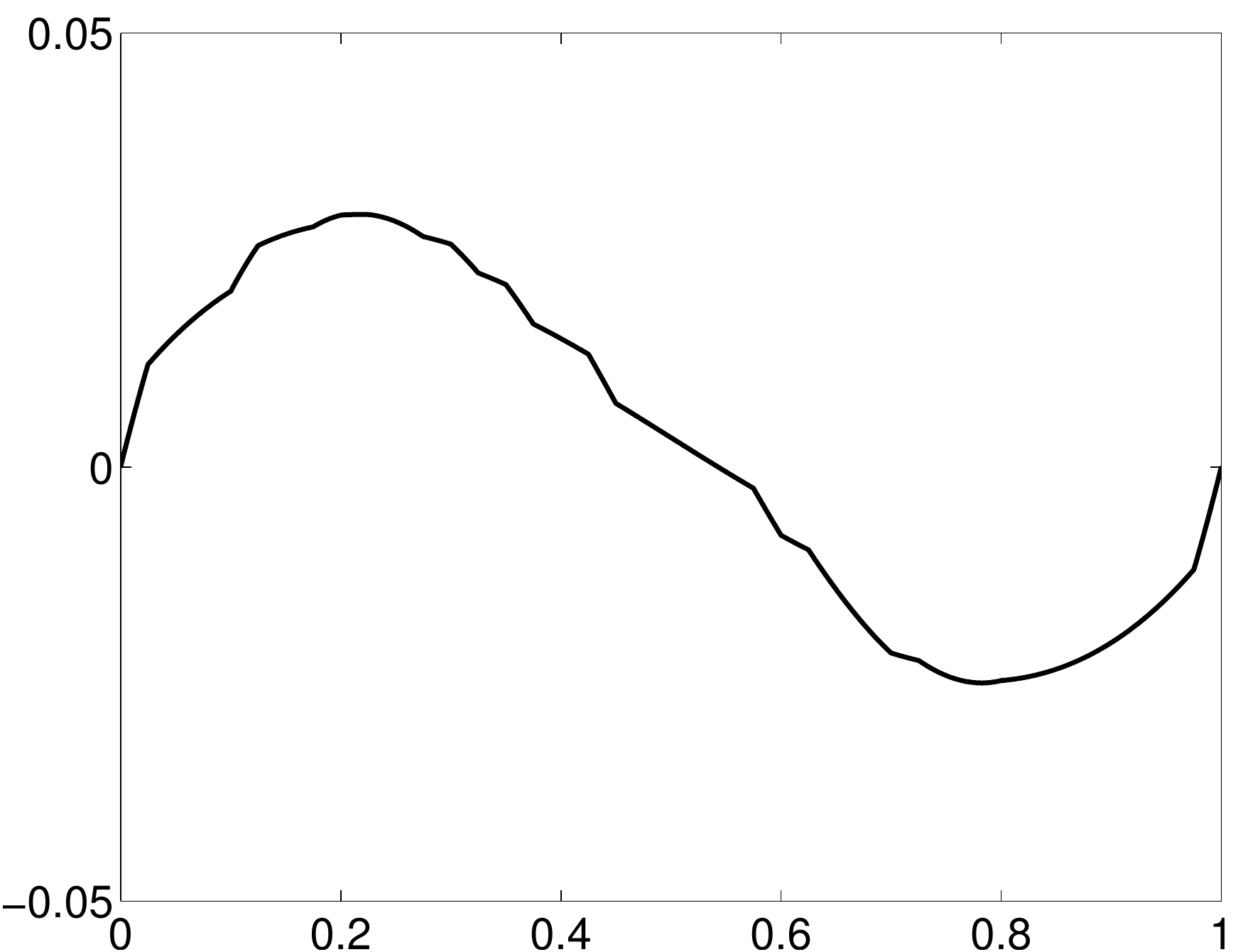}&
\includegraphics[width=.3\textwidth]{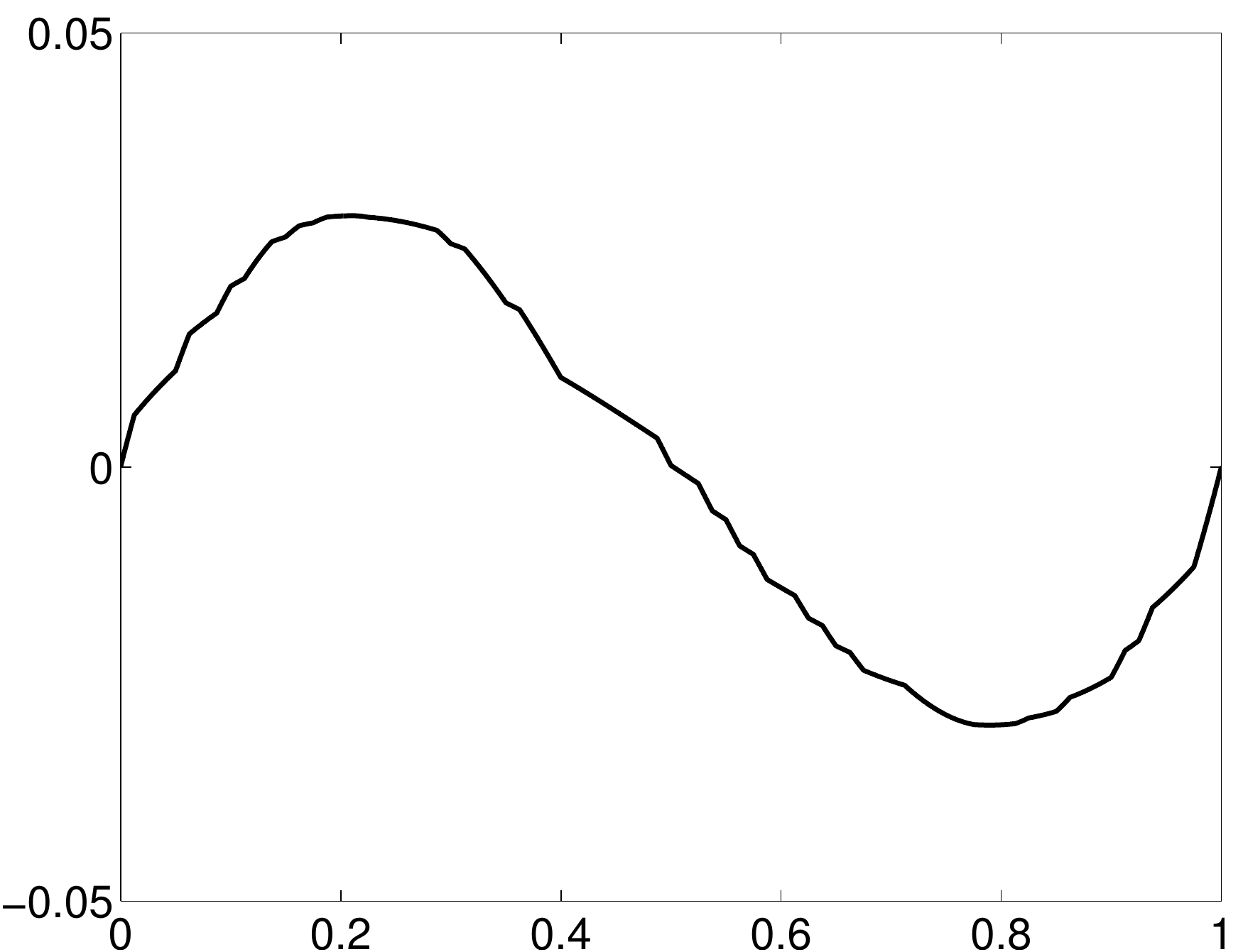}&
\includegraphics[width=.3\textwidth]{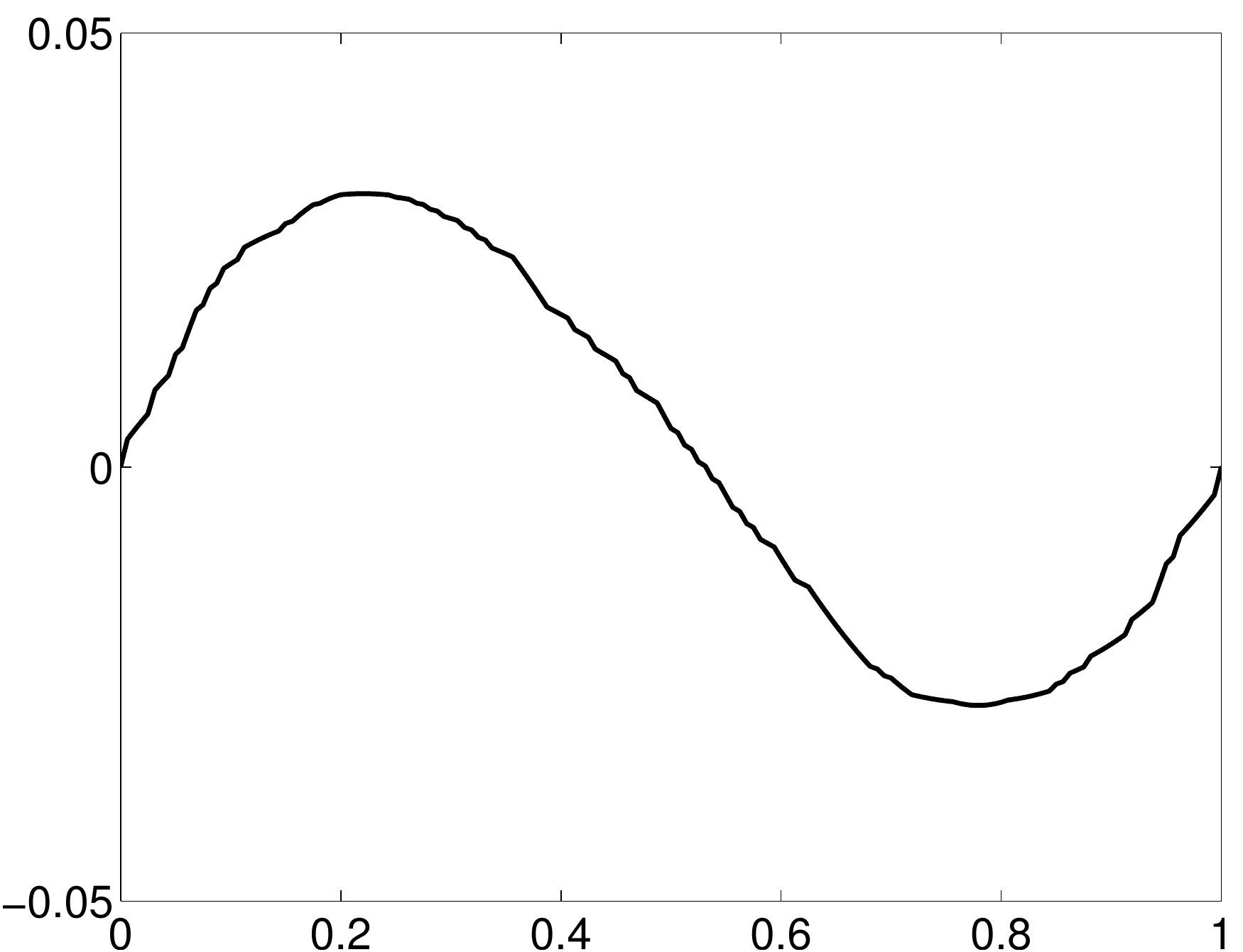}
\end{tabular}
\caption{The solutions $u^\eps$ corresponding to coefficient $a^\eps$ with $\eps = 1/4, 1/8$, and $1/16$ respectively.}
\label{fig:sol_rand}
\end{figure}
In this example, the microstructural length-scale $\eps$ determines the size of the tiles in the
random structure. 

We consider the problem~\eqref{equ:basic-prob-1D-rand} with a fixed realization (a fixed $\omega$) of 
this coefficient function, and for successively smaller values of $\eps$. 
(We continue to use the same right-hand side function $f$ defined in~\eqref{equ:1dprob_data}.)
The solutions $u^\eps(\cdot, \omega)$
of the respective problems have been plotted in Figure~\ref{fig:sol_rand}. 
These plots suggest that $u^\eps$ seems to converge to a limiting function. In what follows, we shall 
discuss the mathematical theory for such stochastic homogenization problems. 
Some relevant questions in this context include the following: (1) is there a homogenized
problem in this stochastic setting? (2) Is it possible to have a 
\emph{constant} homogenized coefficient that is independent of $\omega$? (3) Does the problem admit homogenization 
for all $\omega$? (4) In the deterministic example above periodicity of the coefficient
was the property that led to a constant homogenized coefficient, what is the
stochastic counterpart of periodicity? (5) What conditions on $a(x, \omega)$ ensure 
existence of a deterministic homogenized coefficient? 
A rigorous and clear discussion of such questions, 
which is the main point of this article, requires a systematic synthesis of concepts from functional analysis, 
PDE theory, probability theory, and ergodic theory. 

The discussion rest of this article is structured as follows. 
In Section~\ref{sec:prelim}, we collect the background concepts required in our coverage
of stochastic homogenization.
We continue our discussion by describing the setting of the homogenization problem for random media in Section~\ref{sec:homog-def}. 
Next, in Section~\ref{sec:1dhomog}, we state and prove a homogenization 
theorem in one space dimension. An interesting aspect of the analysis 
for one-dimensional random structures is the derivation of a closed-form expression for the
homogenized coefficient that is analogous to the form of the homogenized coefficient for one-dimensional periodic structures. 
Finally, in Section~\ref{sec:nDhomog}, 
we study homogenization of 
elliptic PDEs with random coefficients in several space dimensions, 
where no closed-form expressions for the homogenized coefficients are available in general.
In Section~\ref{sec:epilogue}, we conclude our discussion by giving some pointers for further
reading.
We mention that an earlier version of the exposition of the theoretical results 
in sections~\ref{sec:1dhomog} and~\ref{sec:nDhomog} appeared first in an introductory chapter 
of the PhD dissertation~\cite{alexanderian2010}. 
 
\section{Preliminaries}\label{sec:prelim}
\subsection{Background from functional analysis and Sobolev spaces}
Here we briefly discuss some background concepts from theory of PDEs and functional 
analysis that are needed in the discussion of the homogenization results in the present work. 
\paragraph{\textbf{Poincar\'{e} inequality}}
Let $\O \subseteq \R^n$ be a bounded open set with piecewise smooth boundary. 
In what follows, we denote by $L^2(\O)$ the space of real-valued square-integrable functions on $\O$ 
and denote by $C^\infty_c(\O)$ the space of smooth functions with compact support in $\O$.
The Sobolev space $H^1(\O)$ consists of functions in $L^2(\O)$ with square integrable 
first-order weak derivatives and is equipped with the norm,
\[
  \Norm{u}{H^1(\O)}^2 = \int_\O u^2\,dx 
                    + \int_\O |\grad u |^2\,dx.
\]
The space $H^1_0(\O)$ is a subspace of $H^1(\O)$ obtained as the closure of $C^\infty_c(\O)$ in $H^1(\O)$.
More intuitively, we may interpret $H^1_0(\O)$ is the subspace of $H^1(\O)$ consisting of functions
in $H^1(\O)$ that vanish on the boundary of $\O$. 
The well-known Poincar\'{e} inequality states that for a 
bounded open set $\O \subseteq \R^n$, there is a positive constant $C_p$ (depending on $\O$ only) 
such that for every $u \in H^1_0(\O)$,
\[
\int_\O u^2 \, dx \leq C_p \int_\O |\grad u|^2 \, dx.
\]

\paragraph{\textbf{Weak convergence}} 
Recall that a sequence
$\{u^k\}_1^\infty$ in a Banach space $X$ converges weakly to $u^* \in X$ 
if $\ell(u^k) \to \ell(u^*)$ as $k \to \infty$, for every bounded linear functional $\ell$ on $X$, 
in which case we write $u^k \weak u^*$.
We recall that, as a consequence of Banach-Steinhaus Theorem, weakly convergent 
sequences in a Banach space are bounded in norm. Moreover, it is a standard result 
in functional analysis that in a reflexive Banach space every bounded 
sequence has a weakly convergent subsequence. Another standard result, which will be used in 
our discussion below, is that compact 
operators on Banach spaces map weakly convergent sequences 
to strongly (norm) convergent sequences. In particular, this implies the 
following: consider a Hilbert space $H$ and a Hilbert subspace $U \subset H$
that is compactly embedded in $H$; then any bounded sequence in $U$ will have a subsequence
that converges strongly in $H$. We also recall that 
in a Hilbert space $H$ with inner-product $\ip{\cdot, \cdot}$, a sequence
$\{u^k\}$ converges weakly to $u^*$ if $\ip{u^k, \phi} \to \ip{u^*, \phi}$ for every $\phi \in H$.

\paragraph{\textbf{Compensated compactness}}
Let $\O$ be a bounded domain in $\R^n$ and suppose 
$\vec{u}^\eps$ converges strongly in $\L^2(\O) = (L^2(\O))^n$ to $\vec{u}^0$ and
$\vec{v}^\eps \weak \vec{v}^0$ in $\L^2(\O)$. In this case, it is straightforward to 
show that $\vec{u}^\eps \cdot \vec{v}^\eps \weak \vec{u}^0 \cdot \vec{v}^0$ in $L^1(\O)$. 
Consider now sequences $\vec{u}^\eps$ and $\vec{v}^\eps$ in $\L^2(\O)$, 
both of which converge weakly. 
In this case, additional conditions are needed 
to ensure the convergence of $\vec{u}^\eps \cdot \vec{v}^\eps$, in an appropriate 
sense, to the inner product of the respective weak limits.
Such problems, which arise naturally in homogenization theory, led to the
development of the concept of compensated compactness by
Murat and Tartar~\cite{Murat:1978,Tartar:1979}. 
Here we recall an important compensated compactness lemma, which specifies conditions 
that enable passing to the limit in the scalar product of weakly convergent sequences and  
concluding the weak-$\star$ convergence 
of the scalar product of the sequences
to the scalar product of their weak limits.
Weak-$\star$ convergence, which is a weaker mode of convergence than weak convergence discussed above,
takes the following form for a sequence of integrable functions:
let $\{z^\eps\}$ be a sequence in $L^1(\O)$, then $z^\eps$ convergences weak-$\star$ to $z^0$ 
if $\{z^\eps\}$ is bounded in $L^1(\O)$ and that, 
\begin{equation*}
   \lim_{\eps \to 0} \int_\O  z^\eps \phi \, dx = \int_\O z^0 \phi \, dx, 
   \quad \forall \phi \in C^\infty_c(\O).
\end{equation*}
We use the notation $z^\eps \weaks z^0$ for weak-$\star$ convergence. The fact that  
weak-$\star$ limits are unique will be important in what follows.

The following Div--Curl Lemma is a well-known compensated compactness result, 
and is a key in proving homogenization results; see~\cite{Kozlov94} for a proof of this
lemma, and~\cite[Chapter 7]{Tartar09} for a more complete discussion as well as interesting
historical remarks on the development the Div--Curl Lemma.
\begin{lemma} \label{lem:compensated}
Let $\O$ be a bounded domain in $\R^n$, and let
$\vec{p}^\eps$ and $\vec{v}^\eps$ be vector-fields in $\L^2(\O)$ 
such that
\begin{equation*}
   \vec{p}^\eps \weak \vec{p}^0, \quad \vec{v}^\eps \weak \vec{v}^0.
\end{equation*}
Moreover assume that $\curl \vec{v}^\eps = 0$ for all $\eps$
and $\div \vec{p}^\eps \to f^0$ in $H^{-1}(\O)$. Then we have
\begin{equation*}
   \vec{p}^\eps \cdot \vec{v}^\eps \weaks \vec{p}^0 \cdot \vec{v}^0.
\end{equation*}
\end{lemma}

\subsection{Background concepts from ergodic theory}
\label{sec:ergodic}
Here we provide 
a brief coverage of the concepts from ergodic theory that are central to the
discussion that follows in the rest of this article. 
We begin by illustrating the concept of ergodicity
through a numerical example. Let $\tor^2$ be the 
two-dimensional unit torus, given by the rectangle $[0,1) \times [0,1)$ with the 
opposite sides identified, and consider the transformation $T: \tor^2 \to \tor^2$ defined by
\begin{equation}\label{equ:cat_map}
   T(\vec{x}) = \begin{bmatrix} (2x_1 + x_2)\!\!\!\!\!\mod 1 \\ (x_1 + x_2)\!\!\mod 1\end{bmatrix}.
\end{equation}
This transformation is an instance of a hyperbolic toral authomorphism~\cite{BrinStuck2002}, 
and is commonly referred to as Arnold's Cat Map, named after V.I.\ Arnold who
illustrated the behavior of the mapping by considering its repeated applications to an image of a cat~\cite{ArnoldAvez68}.

For a given $\vec{x}_0 \in \tor^2$, 
we call the sequence of the points $\{ T^n(\vec{x}_0)\}_{n = 1}^\infty$ the orbit 
of $\vec{x}_0$, where $T^n$ means $n$ successive applications of $T$.  
In Figure~\ref{fig:cat_map}, we depict a portion of the orbit of two different 
points given by $\vec{x}_0 = (1/32, \pi/32)$ 
and $\vec{y}_0 = (1/32, 1/32)$ in the left and right images, respectively. 
\begin{figure}\centering
\includegraphics[width=.35\textwidth]{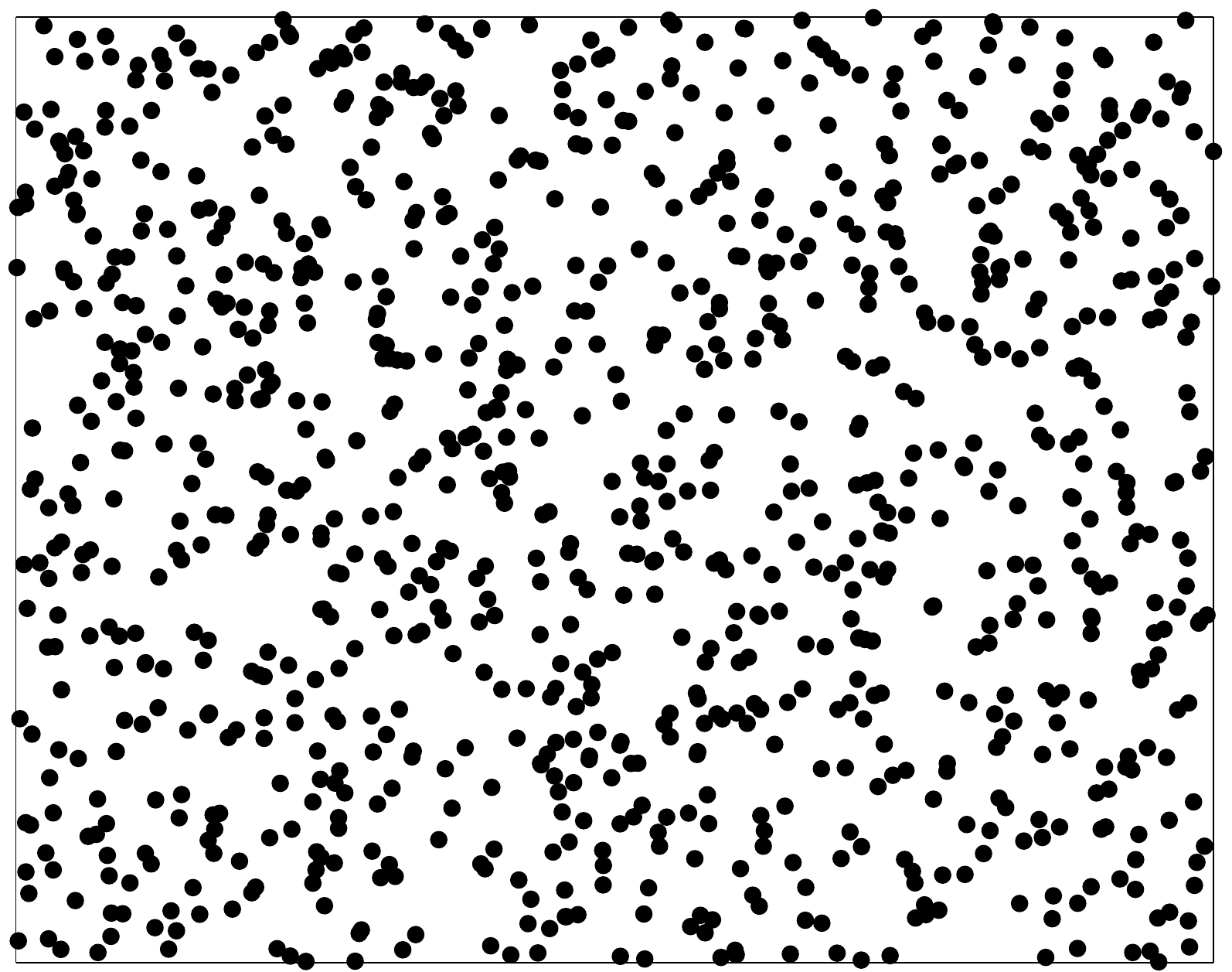}
\includegraphics[width=.35\textwidth]{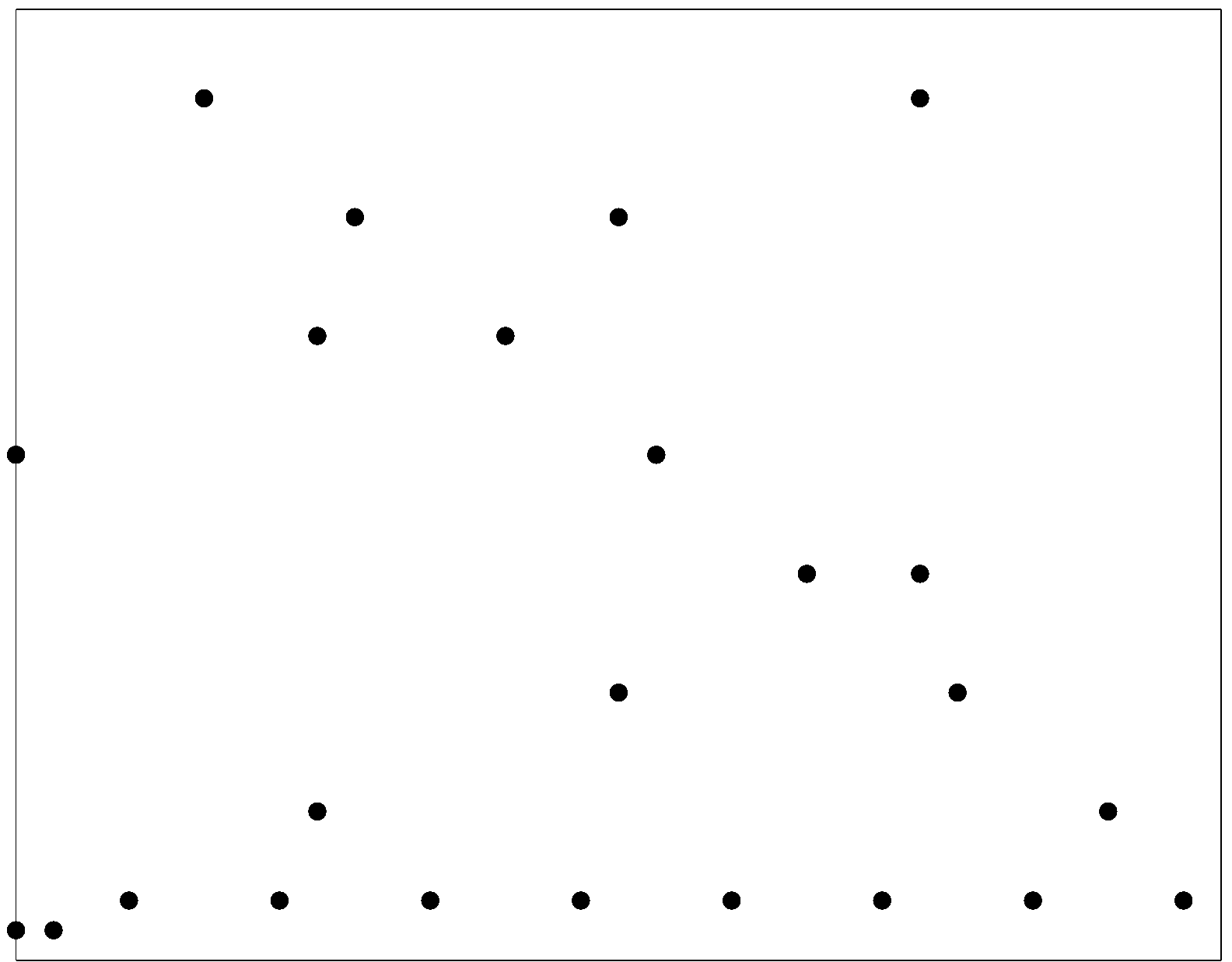}
\caption{For $T$ given in~\eqref{equ:cat_map}, we look at the orbits $\{T^n(\vec{x}_0)\}_{n = 1}^N$ with $\vec{x}_0 = (1/32, \pi/32)$ (left plot) 
and $\{T^n(\vec{y}_0)\}_{n = 1}^N$ with $\vec{y}_0 = (1/32, 1/32)$ (right plot), for $N = 1000$ iterations.}
\label{fig:cat_map}
\end{figure}
The left plot in Figure~\ref{fig:cat_map} suggests that the 
successive iterates $T^n(\vec{x}_0)$ do a good job of visiting the
entire state space $\tor^2$. On the other hand, the plot on the right sends the opposite message.
Note, however, that the coordinates of $\vec{y}_0$ in the latter case are both rational. 
It is known~\cite{BrinStuck2002} that for this specific example the set of points 
with rational coordinates are precisely the set of periodic 
points of the transformation $T$; thus, since the Lebesgue measure of 
this set is zero, we have that \emph{for almost all} $\vec{x}_0 \in \tor^2$, 
the behavior in the left plot of Figure~\ref{fig:cat_map} holds. 
This almost sure ``space filling'' property of the system defined by $T$ is a consequence 
of \emph{ergodicity}. 

Next, consider an integrable function $f:\tor^2 \to \R$. Due to the ``space filling'' property of $T$, 
we may intuitively say that for almost all $\vec{x}_0$ and for $N$ sufficiently large, the 
set of points $\left\{ f\big(T^n(\vec{x}_0)\big)\right\}_{n = 1}^N$ provide
a sufficiently rich sampling of the function $f$ and that 
\[
   \frac1N \sum_{n = 1}^N f(T^n(\vec{x}_0)) \approx \frac1{|\tor^2|}\int_{\tor^2} f(\vec{x}) \, d\vec{x}.
\]
(Here $|\tor^2|$ is the Lebesgue measure of $\tor^2$, which is equal to one, but
is included in the expression for clarity.)
The above observation leads to the usual intuitive understanding of ergodicity:
\emph{for an ergodic system, time averages equal space averages}. 
In the present example, time is specified by $n$, that is 
we have a system with discrete time.

The remainder of this section contains a brief discussion of the concepts
from probability and ergodic theory that we need in our coverage
of stochastic homogenization. For more details on ergodic theory, we refer the reader 
to~\cite{Cornfeld82,Walters82,BrinStuck2002}. See also~\cite{Choe05} for an accessible
introduction to ergodic theory, where the author incorporates many illustrative computer 
examples in the presentation of the theoretical concepts. 

\paragraph{\textbf{Random variables and measure preserving transformations}}
Let $(\Omega, \F, \mu)$ be a probability space. The set $\Omega$ is a sample space,
$\F$ is an appropriate sigma-algebra on $\Omega$, and $\mu$ is a probability measure.
A random variable is a $\F/\B(\R)$ measurable function from $\Omega$ to $\R$, where 
$\B(\R)$ denotes the Borel sigma-algebra on~$\R$. Given a random variable $f:(\Omega,\F, \mu) \to (\R, \B(\R))$,
we denote its expected value by, 
\[
   \ave{f} := \int_\Omega f(\omega) \, \mu(d\omega).
\] 
\begin{definition}
Let $(\Omega_1, \F_1, \mu_1)$ and $(\Omega_2, \F_2, \mu_2)$
be measure spaces. A
transformation $T: \Omega_1 \to \Omega_2$ is called measure preserving if
it is measurable, i.e. for all $E \in \F_2$ 
$T^{-1}(E) \in \F_1$, and satisfies
\begin{equation}\label{equ:meas_pres}
\mu_1\left(T^{-1}(E)\right) = \mu_2(E), \quad \mbox{for all $E \in \F_2$}.
\end{equation}
\end{definition}
\noindent
An example of a measure preserving transformation is
the one defined in~\eqref{equ:cat_map}, which preserves the Lebesgue measure on $\tor^2$. 

\paragraph{\textbf{Dynamical systems and ergodicity}}
Let $T$ be a measure preserving transformation on $(\Omega, \F, \mu)$. 
Interpreting the elements of $\Omega$ as possible states of a system, 
we may consider $T$ as the law of the time evolution of the system.
That is, if we denote by $s_n$, $n \geq 0$, the state of the system at $t = n$, 
and let $s_0 = \omega_0$ for 
some $\omega_0 \in \Omega$,
then, $s_1 = T(\omega_0)$, 
$s_2 = T(s_1) = T(T(\omega_0)) = T^2(\omega_0)$,  and in general,
$s_n = T^n(\omega_0)$, for $n \geq 1$. 
This way, $T$ defines a measurable dynamics on $\Omega$. The dynamical system
so constructed is called a \emph{discrete time} measure-preserving dynamical system.

Suppose there is a set $E \in \F$ such that $\omega \in E$ if and only if $T(\omega) \in E$.
In such a case, the study of the dynamics of $T$ on $\Omega$ can be  
\emph{reduced} to its dynamics on $E$ and $\Omega \setminus E$. 
The set $E$ so described is called a \emph{$T$-invariant} set. We say that $T$ ergodic
if for every $T$-invariant set $E$, we have either $\mu(E) = 0$ or $\mu(E) = 1$.

\paragraph{\textbf{$n$-dimensional dynamical systems}}
In addition to discrete time dynamical systems described above, we can also 
consider \emph{continuous time} dynamical systems that are given by 
a family of measurable transformations $T = \{ T_t\}_{t \in \mathcal{S}}$ where $\mathcal{S} \subseteq \R^n$ with $n = 1$. In 
the case $\mathcal{S} = [0, \infty)$, we call $T$ a \emph{semiflow} and in the case $\mathcal{S} = (-\infty, \infty)$, we call
$T$ a \emph{flow}. In the present work, we are interested in a more general type of dynamical 
system where $\mathcal{S} = \R^n$ with $n \geq 1$.
\begin{definition}
An \emph{$n$-dimensional measure-preserving dynamical system} $T$
on $\Omega$ is a family of measurable 
mappings $T_\vec{x} : \Omega \to \Omega$,
parametrized by $\vec{x} \in \R^n$, satisfying:
\begin{enumerate}
\item 
\(
        T_{\vec{x}+\vec{y}} = T_\vec{x}  \circ T_\vec{y} \text{ for all }
        \vec{x}, \vec{y} \in \R^n.
\)
\item 
\(
        T_{\vec{0}} = I,
\)
where $I$ is the identity map on $\Omega$.
\item
The dynamical system is \emph{measure preserving} in the sense that
for every $\vec{x}\in\R^n$ and $F\in\F$ we have
$\mu\big(T_\vec{x}^{-1}(F)\big) = \mu(F)$.
\item
For every measurable function $g:(\Omega, \F, \mu) \to (X, \Sigma)$ where $(X, \Sigma)$ 
is some measurable space, the composition
$g\big(T_\vec{x}(\omega)\big)$ 
defined on $\R^n \times \Omega$ is a $\big( \B(\R^n) \otimes \F\big)/\Sigma$ measurable 
function.
\end{enumerate}
\end{definition}
The notions of $T$-invariant functions and sets
(where $T$ is an $n$-dimensional dynamical system) are made 
precise in the following definition~\cite{Cornfeld82}.
\begin{definition}\label{def:invariant}
Let $(\Omega, \F, \mu)$ be a probability space and
$\{T_\vec{x}\}_{\vec{x} \in \R^n}$ an n-dimensional measure-preserving 
dynamical system. A measurable function $g$ on $\Omega$ is $T$-invariant if
for all $\vec{x} \in \R^n$,
\begin{equation}\label{equ:inv_func}
   g\big(T_\vec{x} (\omega)\big) = g(\omega), \quad \text{for all }\omega \in \Omega.
\end{equation}
A measurable set $E \in \F$ is $T$-invariant if its
characteristic function $\Char{E}$ is $T$-invariant.
\end{definition}
It is straightforward to show that a $T$-invariant set $E$ defined according to the above definition 
can be defined equivalently as follows: a set $E$ is $T$-invariant if
\begin{equation*}
      T_\vec{x}^{-1}(E) = E, \quad \forall \vec{x} \in \R^n.
\end{equation*}
As is often the case in measure theory, 
we can replace ``for all $\omega \in \Omega$'' by ``for almost all $\omega \in \Omega$'' in Definition~\ref{def:invariant}. 
A function that satisfies~\eqref{equ:inv_func} for all $\vec{x}$ and almost all 
$\omega \in \Omega$ is called $T$-invariant mod 0. 
Also, given two measurable sets $A$ and $B$, we write 
$A = B$ mod 0, if their symmetric difference,
$A\Delta B = (A\setminus B) \cup (B \setminus A)$ has measure
zero; note that this means $A$ and $B$ agree modulo a set of measure zero. 
We call a measurable set $T$-invariant mod 0 if its characteristic function 
is $T$-invariant mod 0.

One can show (cf.~\cite{Cornfeld82}) that for any measurable function $g$ on $\Omega$ that is
$T$-invariant mod 0, there exists a $T$-invariant function $\tilde{g}$ such that $g = \tilde{g}$ almost 
everywhere. Similarly, for any $T$-invariant mod 0 set $E$, there exists a $T$-invariant set $\tilde{E}$
such that $\mu(\tilde{E}\Delta E) = 0$. 
Hence, in what follows, we will not distinguish between $T$-invariance mod 0 and $T$-invariance.

With these background ideas in place, we define the notion of an $n$-dimensional
ergodic dynamical system.
\begin{definition}
Let $(\Omega, \F, \mu)$ be a probability space and
$T = \{T_\vec{x}\}_{\vec{x} \in \R^n}$ an n-dimensional measure-preserving dynamical
system. We say $T$ is ergodic if all $T$-invariant sets have measure of either zero or one.
\end{definition}
\noindent
Let us also recall the following useful characterization
of an ergodic dynamical system \cite{Cornfeld82, Kozlov94}, in terms of
invariant functions:
\textit{a dynamical system is ergodic if every $T$-invariant 
function is constant almost everywhere}; that is,
\begin{equation} \label{equ:ergodic-orbit}
        \Big[
                g\big(T_\vec{x}(\omega)\big) = g(\omega)
                \text{ for all $\vec{x}$ and almost all } \omega
        \Big]
        \quad\Rightarrow\quad
        g \equiv \text{const $\mu$-a.e.}
\end{equation}
Let $\{T_\vec{x}\}_{\vec{x} \in \R^n}$ be a dynamical system.
Corresponding to a function $g : \Omega \to X$ (where $X$ is any set) 
we define the function $g_T : \R^n \times \Omega \to X$ by
\begin{equation} \label{equ:realization-def}
        g_T\big(\vec{x},\omega\big) = g\big(T_\vec{x}(\omega)\big),
        \quad x\in\R^n, \omega\in\Omega.
\end{equation}
For each $\omega \in \Omega$, the function
$g_T\big(\cdot,\omega\big) : \R^n \to X$ is called 
a \emph{realization} of $g$.
\paragraph{\textbf{The Birkhoff Ergodic Theorem}}
Ergodicity of a dynamical system has many profound implications.
Of particular importance to our discussion is the Birkhoff Ergodic Theorem.
Before stating Birkhoff's theorem, we define the following notion of 
mean-value for functions.
\begin{definition}
Let $g \in L^1_\text{loc}(\R^n)$. 
A number $M_g$ is called the mean-value of $g$ if 
for every Lebesgue measurable bounded set $K \subset \R^n$, 
\[
\lim_{\eps \to 0} \frac{1}{|K|} \int_K g(\eps^{-1}\vec{x}) \, d\vec{x} = M_g.
\]
Here $|K|$ denotes the Lebesgue measure of $K$.
\end{definition}
The following result, due to Birkhoff, is a major result
in ergodic theory~\cite{Cornfeld82},
which as we will see shortly, plays
a central role in proving homogenization results for random elliptic operators. 
The statement of Birkhoff's theorem given below follows the presentation in~\cite{Kozlov94}.
\begin{theorem} \label{thm:Birkhoff}
Let $(\Omega, \F, \mu)$ be a probability space, and
suppose $T = \{T_\vec{x}\}_{\vec{x} \in \R^n}$ is a measure-preserving
dynamical system on $\Omega$. Let $g \in L^p(\Omega)$ with
$p \geq 1$. Then for almost all $\omega \in \Omega$ the realization
$g_T(\vec{x}, \omega)$, as defined in~\eqref{equ:realization-def}, 
has a mean value $M_g(\omega)$ in the following sense:
defining $g_T^\eps(\vec{x}, \omega) = g_T(\eps^{-1}\vec{x}, \omega)$ 
for $\eps > 0$, one has
\begin{equation*}
   g_T^\eps(\cdot, \omega) \weak M_g(\omega) \quad \text{ in } L^p_\text{loc}(\R^n), \, \mbox{ as } \eps \to 0,
\end{equation*}
for almost all $\omega \in \Omega$.
Moreover, $M_g$ is a $T$-invariant function; that is,
\begin{equation} \label{equ:invariance}
 M_g\big(T_\vec{x}(\omega)\big) = 
 M_g(\omega) \quad \forall x \in \R^n,~\mu\mbox{-a.e}.
\end{equation}
Also,
\begin{equation} \label{equ:expectation}
   \int_\Omega g(\omega) \, \mu(d\omega) 
      = \int_\Omega M_g(\omega) \, \mu(d\omega).
\end{equation}
\end{theorem}
Notice that if the dynamical system $T$ in Birkhoff's theorem 
is ergodic, then, the mean value $M_g$ is constant almost everywhere 
and is given by $M_g = \ave{g}$.
We record this observation in the following Corollary of Theorem~\ref{thm:Birkhoff}:
\begin{corollary} \label{cor:Birkhoff}
Let $(\Omega, \F, \mu)$ be a probability space, and
suppose $T = \{T_\vec{x}\}_{\vec{x} \in \R^n}$ is a measure-preserving
and \emph{ergodic} dynamical system on $\Omega$. Let $g \in L^p(\Omega)$ with
$p \geq 1$. 
Define $g_T^\eps(\vec{x}, \omega) = g_T(\eps^{-1}\vec{x}, \omega)$ for $\eps > 0$.
Then, for almost all $\omega \in \Omega$
\begin{equation*}
   g_T^\eps(\cdot, \omega) \weak \int_\Omega g(\omega) \, \mu(d\omega)  \quad \text{ in } L^p_\text{loc}(\R^n), \, \mbox{ as } \eps \to 0.
\end{equation*}
\end{corollary}
\paragraph{\textbf{Stationary random fields}}
Let $(\Omega, \F, \mu)$ be a probability space, and
let $G:\R^n \times \Omega \to \R$ be a random field. We say $G$ is
stationary if for any finite collection of points
$\vec{x}_i \in \R^n$, $i = 1, \ldots, k$ and any $\vec{h} \in \R^n$
the joint distribution of the random $k$-vector
$\big(G(\vec{x}_1 + \vec{h}, \omega), \ldots, G(\vec{x}_k + \vec{h}, \omega)\big)^T$ 
is the same as that of $\big(G(\vec{x}_1, \omega), \ldots, G(\vec{x}_k, \omega)\big)^T$.
It is straightforward to show that if $G$ can be written in the form 
\begin{equation} \label{equ:realization}
G(\vec{x}, \omega) = g\big(T_\vec{x}(\omega)\big),
\end{equation}
where $g:\Omega \to \Omega$ is a measurable function
and $T$ is a measure preserving dynamical system, 
then $G$ is stationary. For $G$ to be stationary and ergodic, we need the
dynamical system $T$ in \eqref{equ:realization} to be
ergodic. 

Note that when working with stationary and ergodic random functions, the Birkhoff Ergodic 
Theorem enables the type of averaging that is relevant in the context of 
homogenization. It is also interesting to  
recall the following Riemann-Lebesgue Lemma 
that plays a similar role as Birkhoff's theorem, in the problems 
of averaging of elliptic differential operators 
with \emph{periodic} coefficient functions 
(see~\cite[page 21]{Dacorogna08} for a more general 
statement of the Riemann-Lebesgue Lemma and its proof). 
\begin{lemma}
Let $Y = (a_1, b_1) \times (a_2, b_2) \times \cdots \times (a_n, b_n)$ be a rectangle in $\R^n$ 
and let $g \in L^2(Y)$. Extend $g$ by periodicity from
$Y$ to $\R^n$. For $\eps > 0$, let $g^\eps(\vec{x}) = g(\eps^{-1} \vec{x})$.
Then, as $\eps \to 0$,
$g^\eps \weak \bar{g}$ in $L^2(Y)$, where
$\bar{g} := \frac{1}{|Y|} \int_Y g(\vec{x}) \, d\vec{x}$.
\end{lemma}

\paragraph{\textbf{Solenoidal and potential vector fields and Weyl's decomposition Theorem}}
Let $(\Omega, \F, \mu)$ be a probability space. 
Here we briefly recall an important decomposition of the space $\L^2(\Omega) = L^2(\Omega; \R^n)$ 
of square integrable vector-fields on $\Omega$---the Weyl decomposition Theorem. This result will
be important in homogenization results for random elliptic operators in the general $n$-dimensional 
case.
Recall that a locally square integrable vector-field $\vec v$ on $\R^n$ is called potential if 
$\vec{v} = \grad \phi$ for 
some $\phi \in H^1_\text{loc}(\R^n)$, and is called solenoidal if it is divergence free. 
Letting 
$T$ be an $n$-dimensional measure-preserving dynamical system on $\Omega$, we consider 
the following spaces:
\begin{equation}\label{equ:Spaces}
\begin{aligned}
    \Lpot &= \!\{ \vec{f} \in \L^2(\Omega) : \vec{f}_T(\cdot, \omega) \text{ is
    potential on $\R^n$ for almost all $\omega\in\Omega$} \},\!\!
    \\
    \Lsol &= \!\{ \vec{f} \in \L^2(\Omega) : \vec{f}_T(\cdot, \omega) \text{ is
    solenoidal on $\R^n$ for almost all $\omega\in\Omega$} \},\!\!
    \\
    \Vpot &= \bigl\{ \vec{f} \in \Lpot : \ave{\vec{f}} = \vec{0} \bigr\},\\
    \Vsol &= \bigl\{ \vec{f} \in \Lsol : \ave{\vec{f}} = \vec{0} \bigr\}.
\end{aligned}
\end{equation}
The Weyl Decomposition Theorem  (see e.g.,~\cite[page 228]{Kozlov94}) states that 
the subspaces $\Vpot$ and $\Lsol$ of $\L^2(\Omega)$ 
are mutually orthogonal and complementary, \emph{given that $T$ is ergodic}.
\begin{theorem}[Weyl Decomposition] \label{th:weyl}
If the dynamical system $T$ is ergodic, then $\L^2(\Omega)$ admits 
the following orthogonal decompositions:
\begin{equation}
        \L^2(\Omega)
        = \Vpot \oplus \Lsol
        = \Vsol \oplus \Lpot.
\end{equation}
\end{theorem}

\section{Mathematical definition of homogenization}\label{sec:homog-def}
As before, we let $(\Omega, \F, \mu)$ be a probability space.
The conductivity function of a medium with random microstructure is specified by 
a random function $A(\vec{x}, \omega)$, where for each $\omega \in \Omega$, $A(\cdot, \omega)$ is
a matrix valued function $A(\cdot, \omega):\R^n \to \R^{n\times n}_\text{sym}$. Here 
$\R^{n\times n}_\text{sym}$ denotes the space of symmetric $n \times n$ matrices with real entries.
Let the physical domain be given by a bounded open set
$\O \subset \R^n$ (with $n = 1, 2$, or $3$). Assume for simplicity that the
temperature $u$ is fixed at zero on the boundary of $\O$. 
The PDE governing heat conduction in the medium with microstructure is given by
\begin{equation}  \label{equ:random-problem}
\begin{cases}
-\div_\vec{x}(A(\eps^{-1} \vec{x}, \omega) \grad{u^\eps(\vec{x}, \omega)})=f(\vec{x}) 
  & \mbox{ in }  \O, \\ 
  u^\eps(\vec{x}, \omega) = 0 & \mbox{ on } \partial\O, 
\end{cases}
\end{equation}
where $f \in H^{-1}(\O)$ specifies a (deterministic)
source term. 
The goal of homogenization theory is to specify a problem of 
the form
\begin{equation} \label{equ:hmg-problem}
\left\{ \begin{array}{rcl}
-\div_\vec{x}(A^0 \grad{u^0})=f & 
  \mbox{ in } & \O, \\ 
  u^0 = 0 & \mbox{ on } & \partial\O. 
\end{array}\right. 
\end{equation}
where $A^0$ in~\eqref{equ:hmg-problem} is a constant 
matrix such that the solution $u^0$ of~\eqref{equ:hmg-problem}
provides a reasonable approximation (for almost all $\omega$)
to the solution of~\eqref{equ:random-problem} in the limit as $\eps \to 0$. 
The following definition makes the notion of homogenization precise for a single
deterministic conductivity function.
\begin{definition}\label{def:homog}
Consider a matrix valued function, $A:\R^n \to \R^{n\times n}_\text{sym}$, 
and suppose there exist real numbers $0 < \nu_1 < \nu_2$
such that for each 
$\vec{x} \in \R^n$, 
\[
\nu_1 |\vec{\xi}|^2 \leq \vec{\xi} \cdot A(\vec{x}) \vec{\xi} \leq \nu_2 |\vec{\xi}|^2, \mbox{ for all } \vec{\xi} \in \R^n.
\]
That is, $A$ is uniformly bounded and positive definite. 
For $\eps > 0$, denote $A^\eps(\vec{x}) = A(\eps^{-1}\vec{x})$.
Then, we say that $A$ \emph{admits homogenization} if
there exists a constant symmetric positive definite
matrix $A^0$ such that
for any bounded domain $\O \subset \R^n$ and any
$f \in H^{-1}(\O)$,
the solutions $u^\eps$ of the
problems
\begin{equation} \label{equ:epsproblem}
\left\{ \begin{array}{rcl}
-\div(A^\eps \grad{u^\eps})=f & \mbox{ in } & \O, \\ 
  u^\eps = 0 & \mbox{ on } & \partial\O,
\end{array}\right.
\end{equation}
satisfy the following convergence properties:
\begin{equation*}
   u^\eps \weak u^0 \quad \mbox{ in } H^1_0(\O), \qquad \mbox{ and } \qquad
   A^\eps \grad{u^\eps} \weak A^0 \grad{u^0} \quad \mbox{ in } \L^2(\O),
\end{equation*}
as $\eps \to 0$, where  $u^0$ satisfies the problem
\begin{equation} \label{equ:homog_prob}
\left\{ \begin{array}{rcl}
-\div(A^0\grad{u^0})=f & \mbox{ in } & \O, \\ 
  u^0 = 0 & \mbox{ on } & \partial\O.
\end{array}\right.
\end{equation}
\end{definition}
\noindent
\begin{remark}
In practice, it is sufficient to verify the convergence relations in the above definition for 
right-hand side functions $f \in L^2(\O)$; see also the discussion in~\cite[Remark 1.5]{Kozlov94}.
\end{remark}
\begin{remark}
A family of operators $\{A^\eps\}_{\eps > 0}$ satisfying the above definition are said to 
\emph{G-converge} to $A^0$. The uniqueness of the homogenized matrix $A^0$
is also guaranteed by the uniqueness of G-limits; see e.g.~\cite[page 150]{Kozlov94} or~\cite[page 229]{Hornung1997} 
for basic properties of G-convergence. 
\end{remark}
Note that Definition~\ref{def:homog} concerns the homogenization of a single
conductivity function $A(\vec{x})$. In the case where $A$ is a periodic function,
i.e., the case of periodic media,
the existence of the homogenized matrix is well-known~\cite{
Papa78,
SanchezPalencia,
Olenik92,
CioranescuDonatoBook}.
In the random case~\cite{Kozlov79,
Papa79,
yurinskii1986averaging,
sab-92,
Kozlov94,
Papa95,
blanc2007stochastic}, 
where we work with a random conductivity function $A = A(\vec{x}, \omega)$, 
we say $A$ admits homogenization 
if for almost all $\omega \in \Omega$, $A(\cdot, \omega)$ 
admits homogenization $A^0$ (with $A^0$ a constant matrix independent of 
$\omega$) in the sense of Definition~\ref{def:homog}. 

\section{Stochastic homogenization: the one-dimensional case}
\label{sec:1dhomog}
In this section, we discuss the homogenization of an 
elliptic boundary value problem, in one space dimension, with a random coefficient
function. As we shall see shortly, under
assumptions of stationarity and ergodicity there is
a closed-form expression for the (deterministic) homogenized coefficient. 
Let $(\Omega, \F, \mu)$ be a probability space and 
let $T = \{T_{x}\}_{{x} \in \R}$ be a 1-dimensional 
measure preserving and ergodic dynamical system.
Let $a: \Omega \to \R$ be a measurable function, and suppose  
there exist positive constants $\nu_1$ and $\nu_2$ such that 
\begin{equation} \label{equ:abdd}
      \nu_1 \leq a(\omega) \leq \nu_2, \quad
         \mbox{for almost all  $\omega \in \Omega$}.
\end{equation}
For $\omega \in \Omega$, we consider the following problem,
\begin{equation} \label{equ:basic-prob-1D}
   \begin{aligned} 
      -\ddx \left(a_{\smallT}(\cdot, \omega) \dudx(\cdot, \omega) \right) = f \quad &\mbox{ in }  
      \O = (s, t),  \\
      u(\cdot, \omega) = 0 \quad &\mbox{ on } \partial\O = \{s, t\}.
   \end{aligned}
\end{equation}
Here $\O = (s, t)$ is an open interval, $f \in L^2(\O)$ is a deterministic source term
and $a_{\smallT}(x, \omega) = a\big(T_x(\omega)\big)$ denotes realizations of $a$ with respect to $T$.
Note that by construction, $a_{\smallT}(x, \omega)$ is a stationary and ergodic random field.
\begin{theorem}\label{thm:homog1D}
For almost all $\omega \in \Omega$, $a_T(x, \omega)$ defined above admits homogenization and  
\begin{equation} \label{equ:homog-coeff}
\abar = \frac{1}{\ave{{1}/{a}}}
\end{equation}
is the corresponding homogenized coefficient.
\end{theorem}
\begin{proof}
Since the dynamical system is ergodic, by the Birkhoff Ergodic Theorem, we know
that there is a set
$E \in \F$, with $\mu(E) = 1$ such that for all $\omega \in E$,
\begin{equation} \label{equ:a-conv}
   \frac{1}{a_{\smallT}^\eps(\cdot, \omega)} \weak \ave{\frac{1}{a}} 
   := \frac{1}{\abar}
   \quad
   \mbox{ in }
   L^2(\O),   
\end{equation}
as $\eps \to 0$. Let $\omega \in E$ be fixed but arbitrary and for $\eps > 0$ 
consider the problem
\begin{equation} \label{equ:basic-prob-eps}
   \begin{aligned}  
      -\ddx \left(a_{\smallT}^\eps(\cdot, \omega) \dudxeps(\cdot, \omega) \right) = f
      \quad &\mbox{ in } \O = (s,t),   \\
      u^\eps(\cdot, \omega) = 0 \quad & \mbox{ on } \partial\O = \{s, t\}, 
   \end{aligned}
\end{equation}
with the weak formulation given by,
\newcommand{\dphidx}{\frac{d\phi}{dx}}
\begin{equation} \label{equ:weak}
\int_\O a^\eps_{\smallT}(\cdot, \omega) \dudxeps \dphidx \, dx = \int_\O f \phi\,dx,
\quad \forall \phi \in H^1_0(\O).
\end{equation}

We know that for each $\eps>0$,~\eqref{equ:weak} has a 
unique solution $u^\eps = u^\eps(\cdot, \omega)$. 
First we show that that $\{u^\eps(\cdot, \omega)\}_{\eps > 0}$ is bounded in 
$H^1_0(\O)$ norm. 
To see this, we begin by letting
$\phi = u^\eps$ in \eqref{equ:weak} and note that
\begin{multline*}
   \nu_1 \int_\O \left|\dudxeps\right|^2 \, dx 
   \leq 
   \int_\O a_T^\eps \dudxeps \dudxeps \, dx
   = 
   \int_\O f u^\eps \, dx \\
   \leq 
   \Norm{f}{L^2(\O)}\Norm{u^\eps}{L^2(\O)}
   \leq 
   C_p \Norm{f}{L^2(\O)} \Norm{\dudxeps}{L^2(\O)}, 
\end{multline*}
where the last two inequalities use Cauchy-Schwarz and Poincar\'{e} inequalities respectively.
Thus, 
\begin{equation}\label{equ:derivbdd}
\Norm{\dudxeps}{L^2(\O)} \leq \frac{C_p}{\nu_1} \Norm{f}{L^2(\O)}.
\end{equation}
Moreover, applying Poincar\'{e} 
inequality again, we have 
$\Norm{u^\eps}{L^2(\O)} \leq C_p \Norm{\dudxeps}{L^2(\O)}$ and therefore,
the sequence $\{u^\eps\}$ is bounded in $L^2(\O)$ as well. Thus, we conclude that
$\{u^\eps(\cdot, \omega)\}_{\eps>0}$ 
is bounded in $H^1_0(\O)$. Consequently, we have as $\eps \to 0$,
along a subsequence (not relabeled),
\begin{equation} \label{equ:u-weak}
u^\eps(\cdot, \omega) \weak \ubar \quad \mbox{ in } H^1_0(\O). 
\end{equation} 
Moreover, by compact embedding of $H^1_0(\O)$ into $L^2(\O)$ we have that
$u^\eps(\cdot, \omega) \to \ubar$ strongly in $L^2(\O)$.
Note that at this point it is not clear whether $\ubar$ is independent of $\omega$. 
From \eqref{equ:u-weak} we immediately get that,
\begin{equation} \label{equ:ux-weak}
   \dudxeps(\cdot, \omega) \weak \frac{d\ubar}{dx} \quad \mbox{ in } L^2(\O).
\end{equation}
Next, we let 
\begin{equation}\label{equ:sigma}
\sigma^\eps(x, \omega) = a_{\smallT}^\eps(x, \omega)\dudxeps(x, \omega).
\end{equation}
\newcommand{\dsdxeps}{\frac{d\sigma^\eps}{dx}}

Using the fact that $\{a_{\smallT}^\eps(\cdot, \omega)\}_{\eps > 0}$ 
is bounded in $L^\infty(\O)$ and~\eqref{equ:derivbdd}
we have $\{\sigma^\eps(\cdot, \omega)\}_{\eps > 0}$ is bounded in $L^2(\O)$. 
Moreover, we note that $\displaystyle \dsdxeps = -f$ and 
therefore, $\left\{\dsdxeps(\cdot, \omega)\right\}_{\eps > 0}$ is 
bounded in $L^2(\O)$ as well. 
Therefore, we conclude that $\{\sigma^\eps(\cdot,\omega)\}_{\eps > 0}$ is bounded in $H^1(\O)$.
Thus, $\sigma^\eps(\cdot, \omega) \weak \sbar(\cdot, \omega)$ in $H^1(\O)$ (along a subsequence), and
therefore, by compact embedding of $H^1(\O)$ into $L^2(\O)$ 
we have as $\eps \to 0$, 
\begin{equation} \label{equ:sigma-strong}
   \sigma^\eps(\cdot, \omega) \to \sbar(\cdot,\omega) \quad \mbox{ in } L^2(\O).
\end{equation}
Next, consider the following obvious equality, 
\begin{equation} \label{equ:obv1}
   \dudxeps(\cdot, \omega) = 
      \frac{a_{\smallT}^\eps(\cdot, \omega)}{a_{\smallT}^\eps(\cdot, \omega)}
      \dudxeps(\cdot, \omega) 
      = \sigma^\eps(\cdot, \omega) \frac{1}{a_{\smallT}^\eps(\cdot, \omega)}.
\end{equation}
\noindent
In view of~\eqref{equ:ux-weak} and using~\eqref{equ:a-conv} and \eqref{equ:sigma-strong}
we have as $\eps \to 0$.
\begin{equation*} 
   \sigma^\eps(\cdot, \omega) \frac{1}{a_{\smallT}^\eps(\cdot, \omega)}
   \weak 
   \sbar(\cdot,\omega) \frac{1}{\abar}
   \quad
   \mbox{ in } L^2(\O), \qquad \mbox{ and } \qquad \frac{d\ubar}{dx} = \sbar(\cdot,\omega) \frac{1}{\abar}.
\end{equation*}
Thus, we have 
\[
\sbar(\cdot,\omega) = \abar \frac{d\ubar}{dx},
\] 
and,
recalling the definition of $\sigma^\eps$ in~\eqref{equ:sigma}, we can rewrite~\eqref{equ:sigma-strong}
as follows:
\begin{equation}\label{equ:fluxconv}
 a_{\smallT}^\eps(x, \omega)\dudxeps(x, \omega) \to  \abar \frac{d\ubar}{dx}, \quad \mbox{ in } L^2(\O).
\end{equation}
Hence, passing to the limit as $\eps \to 0$ in~\eqref{equ:weak} gives, 
\[
\int_\O  \abar \frac{d\ubar}{dx} \dphidx \,dx = \int_\O f \phi\,dx,
\quad \forall \phi \in H^1_0(\O),
\]
which says that $\ubar$ is the weak solution to
\begin{equation} \label{equ:basic-prob-homog}
 \begin{aligned}
      -\ddx \left(\abar \frac{d\ubar}{dx} \right) = f \quad
      &\mbox{ in }  
      \O = (s,t),   \\
      \ubar = 0 \quad &\mbox{ on } \partial\O = \{s, t\}. 
 \end{aligned} 
\end{equation}
Note also that by~\eqref{equ:abdd} we have  
that $\nu_1 \leq \abar \leq \nu_2$. 
The problem~\eqref{equ:basic-prob-homog} has a 
unique solution $\ubar$ that is independent of $\omega$,  because $\abar$ is 
a constant independent of $\omega$ and the right-hand side function $f$
is deterministic. Also, since the solution $\ubar$ is unique,  
any subsequence of $u^\eps(\cdot, \omega)$ converges to the same limit $\ubar$
(weakly in $H^1_0(\O)$ and thus strongly in $L^2(\O)$) and thus the entire sequence $\{u^\eps(\cdot, \omega)\}_{\eps > 0}$ 
converges to $\ubar$, not just such a subsequence.
Finally, since the domain $\O$ was any arbitrary open interval and 
the right-hand side function $f \in L^2(\O)$ was arbitrary,~\eqref{equ:u-weak},~\eqref{equ:fluxconv} 
and~\eqref{equ:basic-prob-homog} lead to the conclusion 
that $a_{\smallT}^\eps(\cdot, \omega)$ admits homogenization with homogenized coefficient given 
by $\abar = \ave{1/a}^{-1}$. Note also that this conclusion holds for almost all $\omega \in \Omega$.
\end{proof}
\begin{remark}
Note that Theorem~\ref{thm:homog1D} says the effective coefficient $\abar$ is a 
constant function on $\O$ with $\abar(x) = \ave{{1}/{a}}^{-1}$ for all $x \in \O$.
Also, observe that $\abar$ is the one-dimensional counterpart 
of the homogenized coefficient $A^0$ in~\eqref{equ:homog_prob}.
\end{remark}

\section{Stochastic homogenization: the $n$-dimensional case}
\label{sec:nDhomog}
Before delving into the theory, we consider a numerical illustration of
homogenization in a two-dimensional example.
We consider,
\begin{equation} \label{equ:test} 
\left\{ \begin{array}{rcl}
-\div(A(\eps^{-1}\vec{x}, \omega) \grad{u^\eps(\vec{x}, \omega)})=f(\vec{x}) &
  \mbox{ in } & \O = (0,1)\times(0,1), \\
  u^\eps(\vec{x}, \omega) = 0 & \mbox{ on } & \partial\O, 
\end{array}\right.
\end{equation}
where the source term is given by,
\[
   f(\vec{x}) = \frac{C}{2\pi L} \exp\left\{ -\frac{1}{2L}\big[ (x_1 - 1/2)^2 + (x_2 - 1/2)^2\big] \right\}, 
\quad \mbox{ with } C = 5, \mbox{and } L = .05.
\] 
We describe the diffusive properties of the medium, modeled by the conductivity function $A(\vec{x}, \omega)$, 
by a random tile based structure
similar to the one-dimensional example presented in the beginning 
of the article.  Consider a checkerboard like structure where the conductivity of each 
tile is a random variable that can take four possible values $\kappa_1, \ldots, \kappa_4$, with 
probabilities
$p_i \in (0,1)$, $\sum_{i = 1}^4 p_i = 1$.
For the present example, we let $\kappa_1 = 1$, $\kappa_2 = 10$, $\kappa_3 = 50$, and $\kappa_4 = 100$, 
which can occur with probabilities $p_1 = 0.4$, and $p_2 = p_3 = p_4 = 0.2$, 
respectively.  We depict a realization of the resulting (scalar-valued) random conductivity function $A(\vec{x}, \omega)$ in 
Figure~\ref{fig:sample1} (left) and the solution $u(\vec{x},\omega)$ of the corresponding diffusion problem~\eqref{equ:test} in 
the right image of the same figure.
Note that in the plot of the random checkerboard, lighter colors correspond to tiles with larger conductivities.
\begin{figure}[ht]\centering
\includegraphics[width=.25\textwidth]{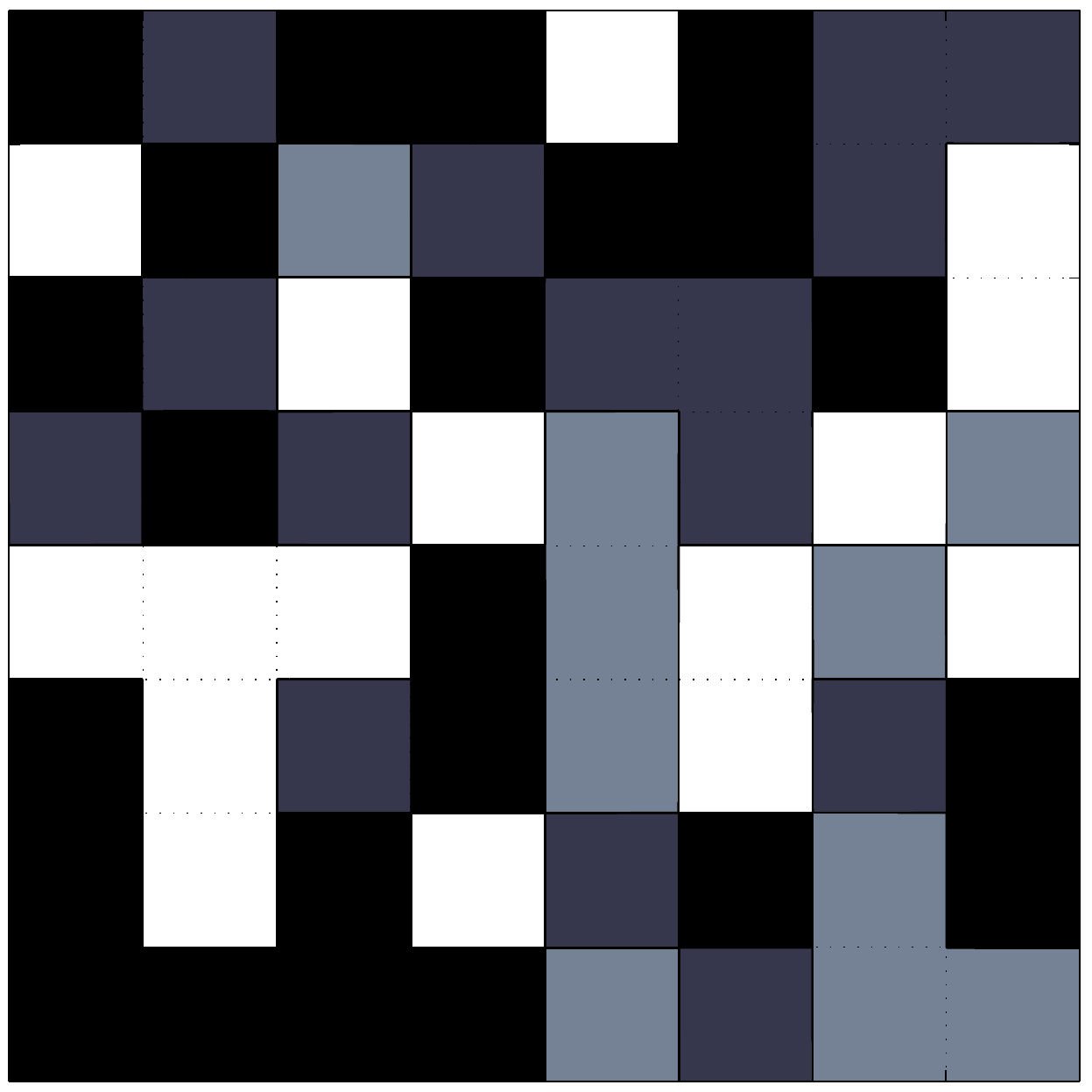}
\includegraphics[width=.252\textwidth]{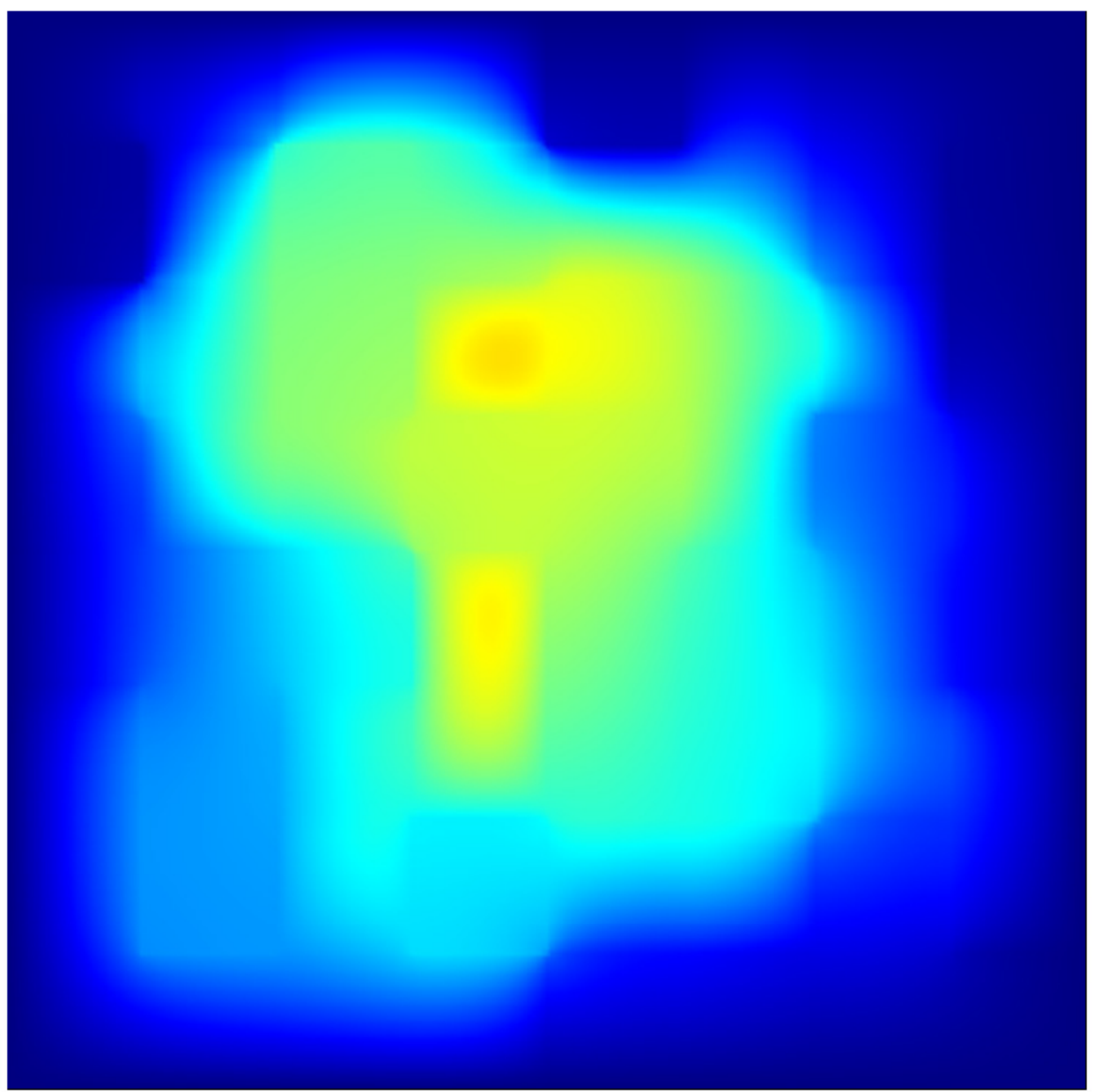}
\includegraphics[width=.053\textwidth]{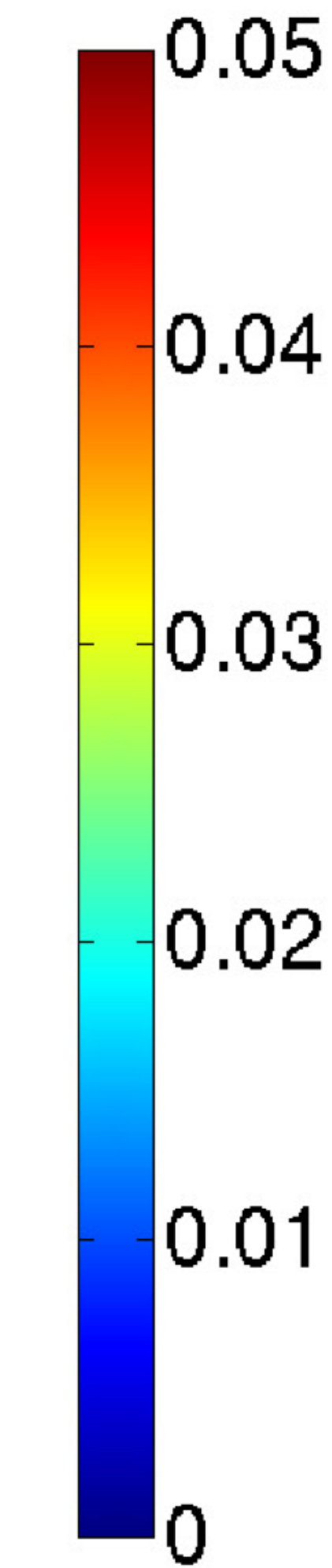}
\caption{Left: a realization of the random checkerboard conductivity function described above; right: the
solution $u(\vec{x},\omega)$ corresponding to the realization of $A(\vec{x}, \omega)$.}
\label{fig:sample1}
\end{figure}
For a numerical illustration of homogenization, we compute the solutions of problem~\eqref{equ:test} with successively 
smaller values of $\eps$. Specifically, using the same realization of the medium shown in Figure~\ref{fig:sample1} (left), 
we solve the problem~\eqref{equ:test} with $\eps = 1/2, 1/4$, and $1/8$.
Results are reported in Figure~\ref{fig:homog2d}, where we plot 
the coefficient fields $A(\eps^{-1}\vec{x}, \omega)$ (top row) 
and the corresponding solutions $u^\eps(\vec{x}, \omega)$ (bottom row).
Note that as $\eps$ gets smaller the solutions $u^\eps$ seem to approach that of a diffusion problem
with a constant diffusion coefficient. This is the expected outcome when working with structures that 
admit homogenization.
We mention that these problems were solved numerically using a continuous Galerkin finite-element
discretization with a $200 \times 200$ mesh of quadratic quadrilateral elements. COMSOL Multiphysics
was used for the finite-element discretization and computations were performed in Matlab.

\begin{figure}[ht]\centering
\includegraphics[width=.25\textwidth]{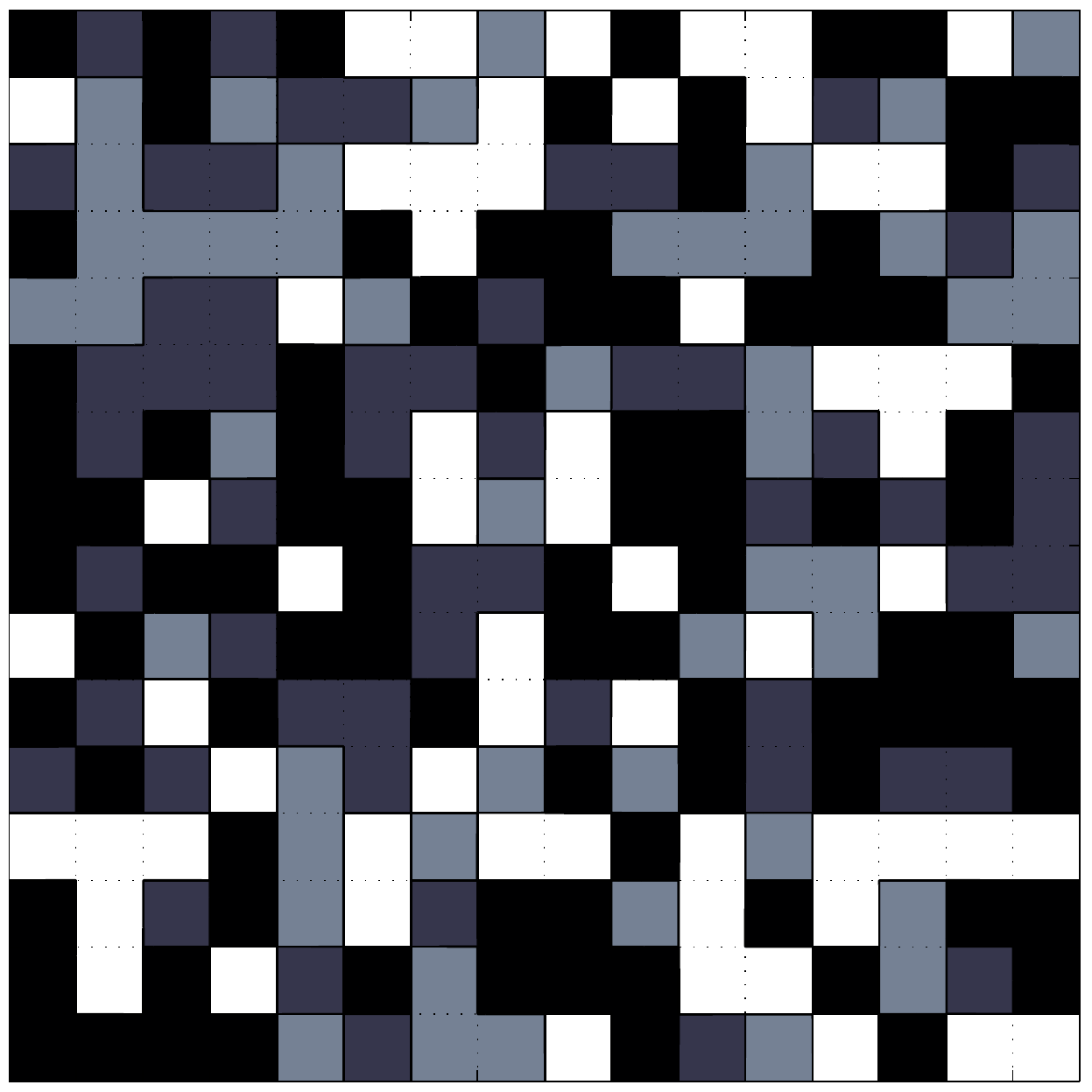}
\includegraphics[width=.25\textwidth]{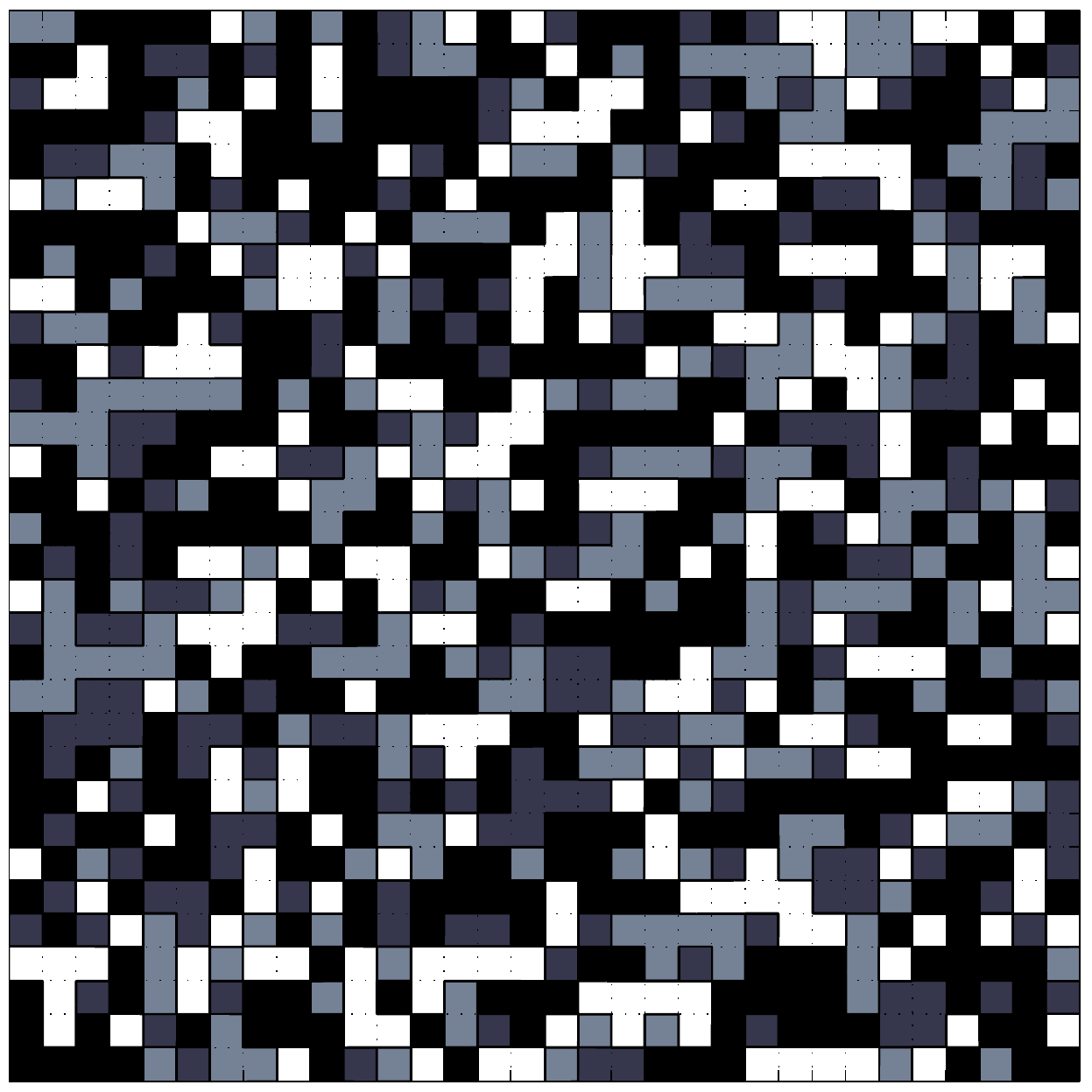}
\includegraphics[width=.25\textwidth]{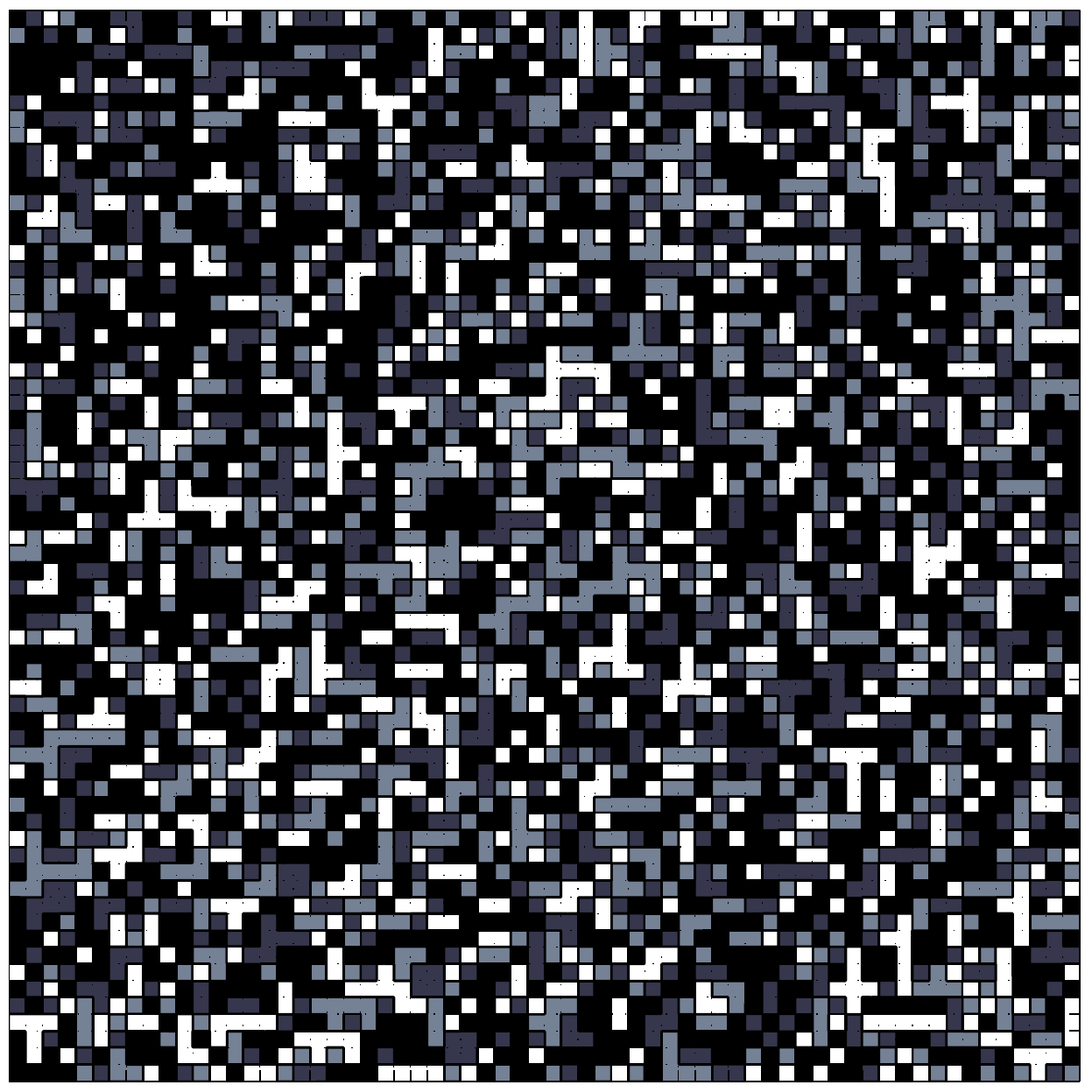}\\
\hspace{6mm}
\includegraphics[width=.25\textwidth]{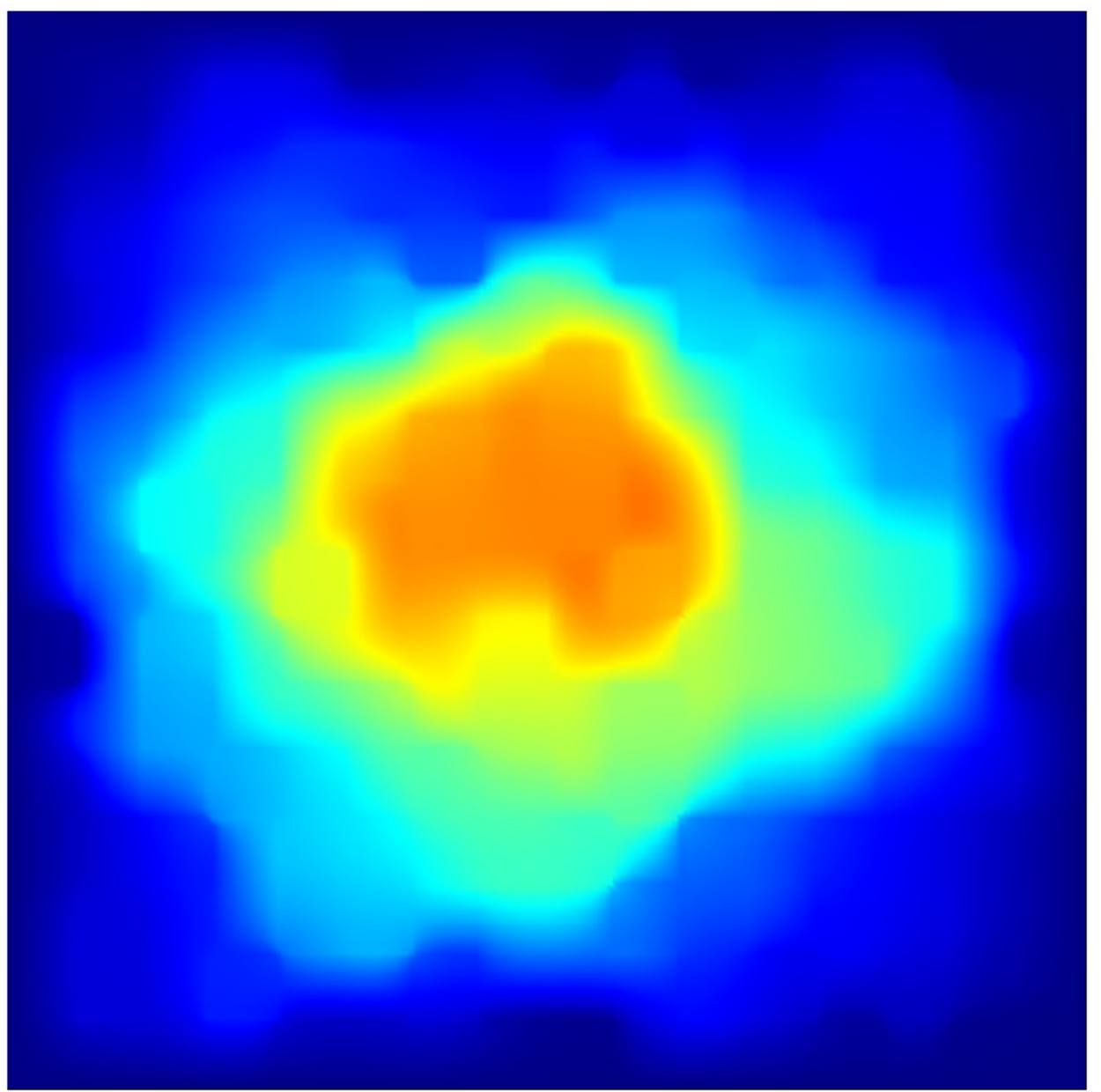}
\includegraphics[width=.25\textwidth]{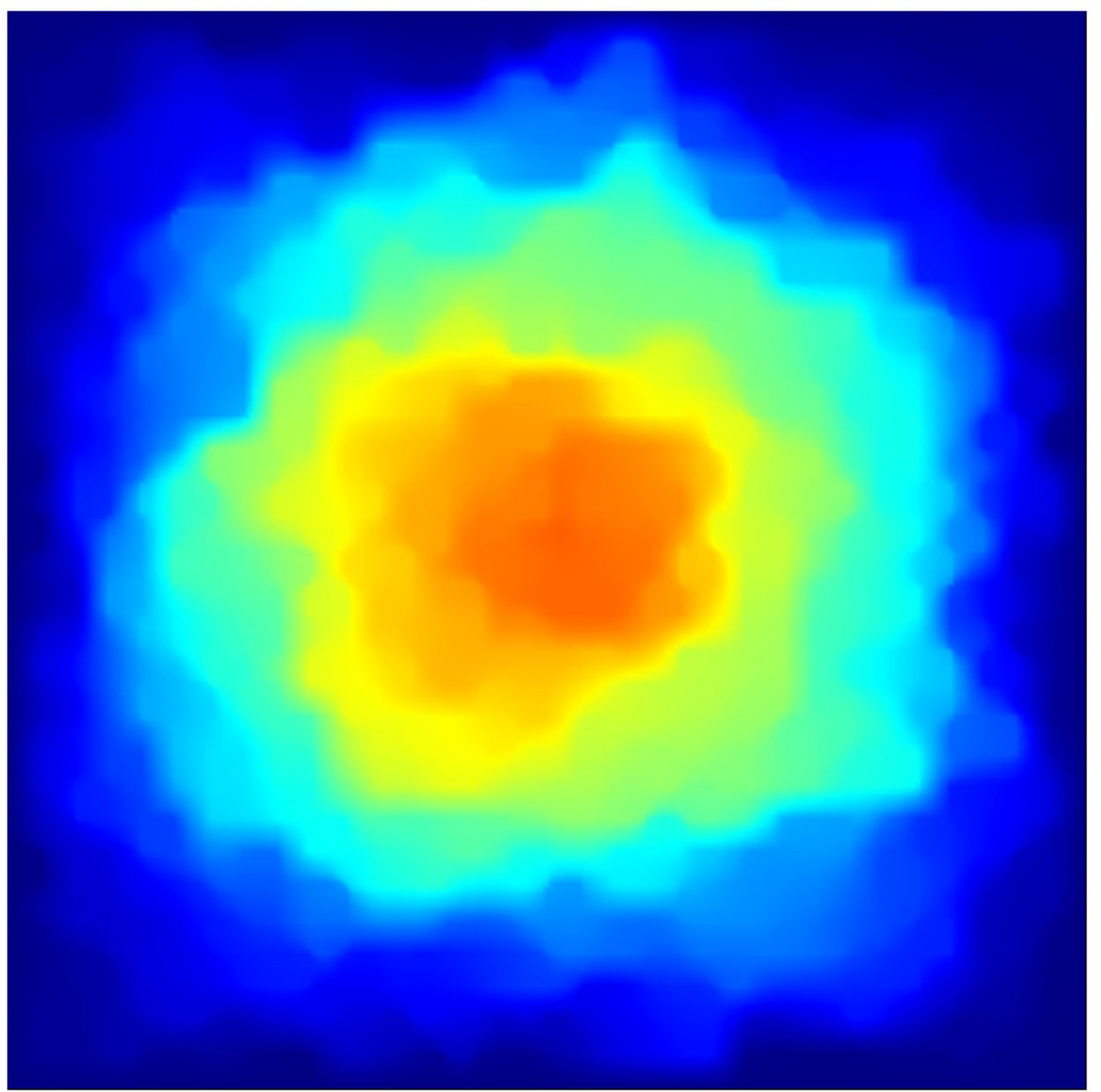}
\includegraphics[width=.25\textwidth]{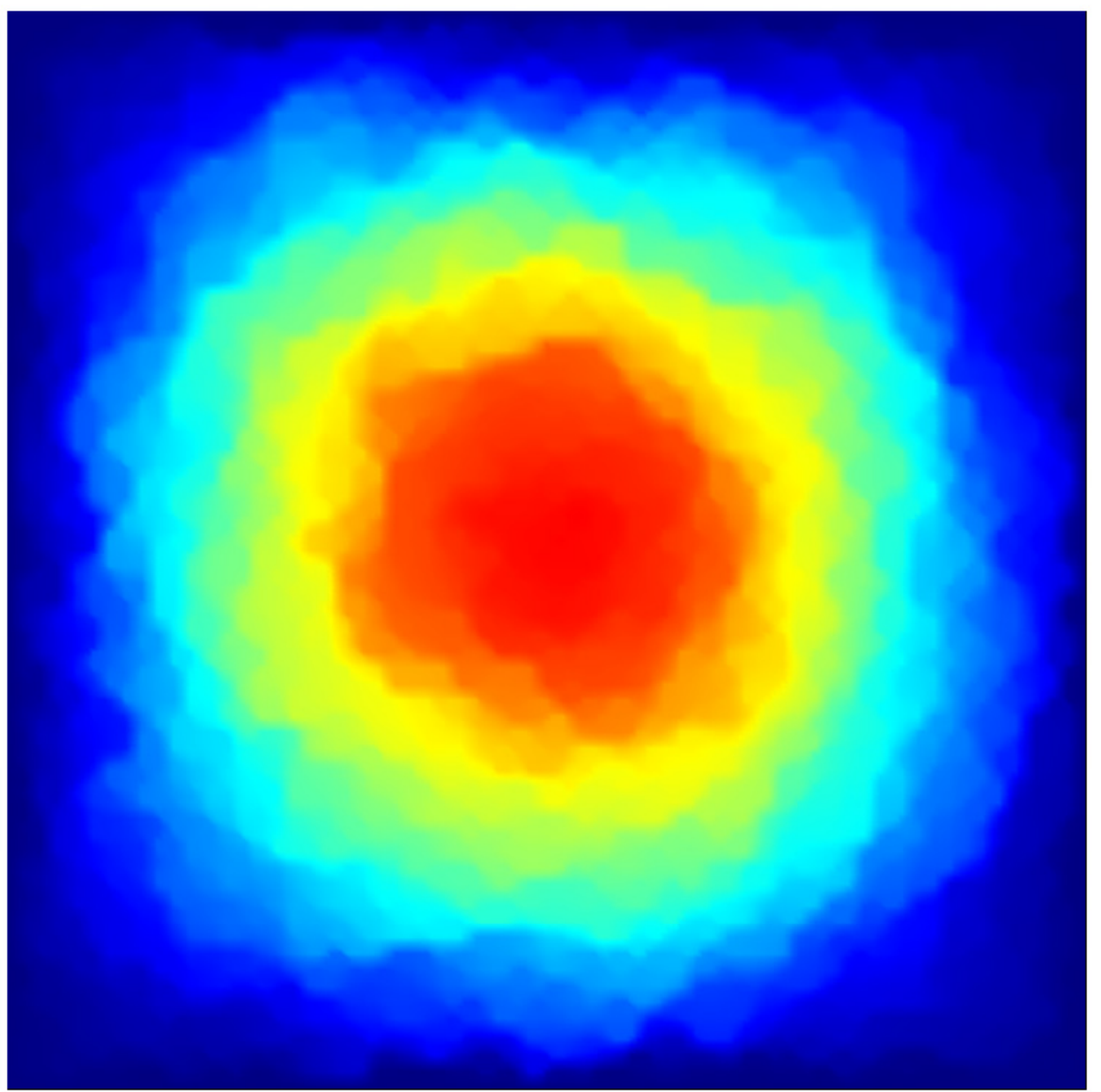}
\includegraphics[width=.05\textwidth]{colbar.pdf}
\caption{Top row: $A(\eps^{-1}\vec{x},\omega)$, for a fixed $\omega$, 
with $\eps = 1/2, 1/4$, and $1/8$; bottom row, the respective solutions $u^\eps(\vec{x}, \omega)$.}
\label{fig:homog2d}
\end{figure}

Below, we study a homogenization result in $\R^n$, which shows that under assumptions of
stationarity and ergodicity, a homogenized medium exists. 
As we shall see shortly, in this general $n$-dimensional case, unlike the
one-dimensional problem, there is no closed-form analytic formula for the homogenized coefficients. (Analytic
formulas for the homogenized coefficients are available only in some special cases in two dimensions~\cite{Kozlov94}.) 
Note that even in the case of periodic structures in several space dimensions, 
analytic formulas for the homogenized coefficient are not available; however, in the periodic
case, the characterization of the effective coefficients suggests a straightforward computational 
method for computing the homogenized conductivity matrix. This is no longer the case 
in the stochastic case, where the numerical approximation of homogenized coefficients is 
generally a difficult problem; see also Remark~\ref{rmk:periodization} below. 

\subsection{The  homogenization theorem in $\R^n$}
In this section, we present the stochastic homogenization theorem for linear elliptic operators 
in $\R^n$. The discussion in this section follows in similar lines as that presented 
in~\cite{Kozlov94}. Consider the problem, 
\begin{equation}  \label{equ:ergodic-problem}
\left\{ \begin{array}{rcl}
-\div(A(\eps^{-1}\vec{x}, \omega) \grad{u^\eps(\vec{x}, \omega)})=f(\vec{x}) & 
  \mbox{ in } & \O, \\ 
  u^\eps(\vec{x}, \omega) = 0 & \mbox{ on } & \partial\O.
\end{array}\right. 
\end{equation}
Here $\O$ is a bounded domain in $\R^n$, $f \in L^2(\O)$ is a deterministic source 
term, and $A$ is a stationary and ergodic random field. That is, we assume that
\begin{equation}
A(\vec{x}, \omega) = \A(T_\vec{x} (\omega)), \quad \forall \vec{x} \in \R^n, 
                                                         \,\omega \in \Omega,
\end{equation}
where $T = \{T_\vec{x}\}_{\vec{x} \in \R^n}$ is an $n$-dimensional 
measure preserving and ergodic dynamical system, and $\A$ is a measurable function
from $\Omega$ to $\R^{n \times n}_\text{sym}$ that is uniformly bounded and positive definite. 
We define the set of all such $\A$ as follows. For positive constants $0 < \nu_1 \le \nu_2$ let 
\begin{multline*}
        \Es(\nu_1,\nu_2, \Omega)
        =
        \{ \A:\Omega \to \Rnnsym : \quad
                \A \text{ is measurable and } \\ 
                \nu_1|\vec{\xi}|^2
                \le \vec{\xi} \cdot \A(\omega) \vec{\xi}
                \le \nu_2 |\vec{\xi}|^2\quad\forall \vec{\xi} \in \R^n,
                \,\text{for almost all } \omega \in \Omega\}.
\end{multline*}
\noindent
Note that here $|\cdot|$ denotes the Euclidean norm in $\R^n$; i.e., $|\vec{\vec{\xi}}|^2 = \sum_{i = 1}^n\xi_i^2$.
The following homogenization result (cf.~\cite[Theorem 7.4]{Kozlov94}) 
provides a characterization of the homogenized matrix for stationary and
ergodic diffusive media.
\begin{theorem} \label{thm:homogthm-main}
Let $\A:\Omega \to \R^{n \times n}$ be in $\Es(\nu_1, \nu_2, \Omega)$
for some $0 < \nu_1 \leq \nu_2$.
Moreover, assume that $T = \{T_\vec{x}\}_{\vec{x} \in \R^n}$ is a measure preserving and ergodic dynamical
system. Then, for almost all $\omega\in\Omega$,
the realization $\A_T(\cdot, \omega)$ admits homogenization, and the homogenized 
matrix $\A^0$ is characterized by,
\begin{equation} \label{equ:A0_char}
        \A^0 \boldsymbol\xi  = \int_\Omega \A(\omega) \big(\boldsymbol\xi +
        \vec{v}_{\boldsymbol\xi}(\omega)\big) \, \mu(d\omega), \quad \forall \boldsymbol\xi \in
\R^n,
\end{equation}
where $\vec{v}_{\boldsymbol\xi}$ is the solution to the
following auxiliary problem:
Find $\vec{v} \in \Vpot$ (recall the definition of $\Vpot$ in~\eqref{equ:Spaces})
such that
\begin{equation} \label{equ:auxprob}
   \int_\Omega \A(\omega)\big(\boldsymbol\xi + \vec{v}(\omega)\big) \cdot \vec{\varphi}(\omega) \, \mu(d\omega) = 0, 
   \quad \forall \vec{\varphi} \in \Vpot.
\end{equation}
\end{theorem}

Before presenting the proof of this result, we collect some observations.
\begin{remark}\label{rmk:periodization}
Note that Theorem~\ref{thm:homogthm-main} provides an abstract characterization for $\A^0$, which
does not lend itself directly to a numerical recipe for computing $\A^0$. 
While the discussion in the present note does not include numerical methods,
we point out that numerical approaches for computing $\A^0$ are available.
See e.g.,~\cite{Owhadi02,Bourgeat04} that describe the method of periodization, 
which can be used to compute approximations to the homogenized matrix $\A^0$. 
\end{remark}

\begin{remark}
The above homogenization result applies to random diffusive media whose conductivity functions are
described by stationary and ergodic random fields.
From a practical point of view, 
such ergodicity assumptions are mathematical niceties that cannot be verified in real-world problems. 
One possible idea is to construct 
mathematical definitions of certain ``idealized'' random structures
for which one can prove ergodicity and use such structures as potential 
modeling tools in real applications. An example of such an effort 
is done in~\cite{AlexanderianRathinamRostamian12}, 
where, starting from their physical descriptions, a class of stationary and 
ergodic tile-based random structures has been constructed. 
See also the book~\cite{Torquato}, which provides a comprehensive
treatment of means for statistical characterization of random heterogeneous materials. 
\end{remark}

\begin{remark}
The form of the homogenized coefficient in one space dimension given by 
Theorem~\ref{thm:homog1D} 
can be derived by specializing Theorem~\ref{thm:homogthm-main} to the case of $n = 1$. 
To see this, we note that in the one-dimensional case, the homogenized coefficient is characterized as follows: 
For $\xi \in \R$,
\begin{equation} \label{equ:abar-char}
   \azero \xi = \int_\Omega a(\omega)(\xi + v_\xi(\omega)) \, \mu(d\omega),
\end{equation}
where $v_\xi \in \Vpot$ is solution to the auxiliary problem~\eqref{equ:auxprob}. Hence, using Weyl's theorem, 
we may write,
\begin{equation} \label{equ:aux}
   a\,(\xi + v_\xi) \in \Lsol.
\end{equation}
To find $\azero$ we need to only consider $\xi = 1$ in~\eqref{equ:abar-char}. Denote,
\begin{equation} \label{equ:q}
 q(\omega) = a(\omega)(1 + v_1(\omega)),
\end{equation}
and note that by \eqref{equ:aux}, and recalling the definition of $\Lsol$, we have that 
for almost all $\omega$, $q\big(T_x(\omega)\big)$ is a
constant (depending on $\omega$). That is, for almost all $\omega \in \Omega$,
$q\big(T_x(\omega)\big) = q(\omega)$, for all $x \in \R$.
Therefore, by ergodicity of the dynamical system $T$, we
have
$q(\omega) \equiv \mathrm{const} =: \bar{q}$ almost everywhere.
Thus, using \eqref{equ:q} we have 
$v_1(\omega) = {\bar{q}}/{a(\omega)} - 1$, and since
$\ave{v_1} = 0$, we have $\bar{q} = \ave{1/a}^{-1}$.
Then, \eqref{equ:abar-char} gives 
\begin{equation*}
\azero = \int_\Omega a(\omega)(1 + v_1(\omega)) \, \mu(d\omega) 
      = \int_\Omega \bar{q} \, \mu(d\omega)
      = \bar{q} 
      = \ave{1/a}^{-1}.
\end{equation*}
\end{remark}

Next, we turn to the proof of Theorem~\ref{thm:homogthm-main}:

\begin{proof}
First we note that the characterization of $\A^0$ in the statement of the theorem along with the properties of $\A$ allows
us to, through a standard argument, conclude that $\A^0$ is a symmetric positive definite matrix
(see Section~\ref{sec:variational} for a proof of this fact). 
Consider the family of Dirichlet problems
\begin{equation*}
         \left\{ \begin{array}{lcl}
         -\div\big(\A_T^\eps(\vec{x},\omega) \grad{u^\eps(\vec{x},
\omega)}\big) = f(\vec{x}) & 
         \mbox{ in } & \O,  \\
         u^\eps(\vec{x}, \omega) = 0 & \mbox{ on } & \partial \O,
         \end{array}\right.
\end{equation*}
whose weak formulation is given by,
\begin{equation}\label{equ:weakform_nd}
\int_\O \A_T^\eps(\cdot, \omega) \grad{u^\eps(\cdot, \omega)} \cdot \grad{\phi} \, d\vec{x} 
= \int_\O f \phi \, d\vec{x}, \quad \forall \phi \in H^1_0(\O).
\end{equation}
For a fixed $\omega$, we can use arguments similar to 
those in the one-dimensional case, to show that the family of functions 
$u^\eps(\cdot, \omega)$ is
bounded in $H^1_0(\O)$ and the family of functions 
$\vec{\sigma}^\eps(\cdot, \omega) = \A_T^\eps(\cdot, \omega)\grad{u^\eps(\cdot, \omega)}$
is bounded in $\L^2(\O)$. Therefore, 
(along a subsequence) as $\eps \to 0$ 
\begin{eqnarray}
   u^\eps(\cdot, \omega) \weak u^0 \quad \mbox{ in } H^1_0(\O),\label{equ:u-weak_nd} \\
   \vec{\sigma}^\eps(\cdot, \omega) = \A_T^\eps(\cdot, \omega) 
                                 \grad{u^\eps(\cdot, \omega)} \weak \vec{\sigma}^0 
   \quad \mbox{ in } \L^2(\O).\label{equ:fluxconv_nd}
\end{eqnarray}
Note that~\eqref{equ:u-weak_nd} also implies that $\grad{u^\eps(\cdot, \omega)} \weak \grad{u^0}$ in $\L^2(\O)$.
Our goal is to show that $\vec{\sigma}^0 = \A^0 \grad{u^0}$ and
that the limit $u^0$ is the (weak) solution of the problem
\begin{equation}\label{equ:homogenized_pde}
\left\{ \begin{array}{rcl}
-\div(\A^0\grad{u^0})=f & \mbox{ in } & \O, \\ 
  u^0 = 0 & \mbox{ on } & \partial\O.
\end{array}\right.
\end{equation}
Let $\vec{\xi} \in \R^n$ be fixed but arbitrary and let $\vec{p} = \vec{p}_\vec{\xi}$ be given by,
\begin{equation}
   \vec{p} = \vec{\xi} + \vec{v}_\vec{\xi},
\end{equation}
where $\vec{v}_\vec{\xi} \in \Vpot$ solves~\eqref{equ:auxprob}.
Note that $\vec{p} \in \Lpot$ with $\ave{\vec{p}} = \vec{\xi}$.
Moreover, let $\vec{q}(\omega) = \A(\omega)\vec{p}(\omega)$ and note that 
\[
   \ave{\vec{q}} = \int_\Omega \A(\omega) \vec{p}(\omega) \, \mu(d\omega) 
   = \int_\Omega \A(\omega) \big(\vec{\xi} + \vec{v}_\vec{\xi}(\omega)\big) \, \mu(d\omega)
   = \A^0 \vec{\xi},
\]
where the last equality follows from~\eqref{equ:A0_char}. 
Moreover, let us note that since $\vec{v}_\vec{\xi}$ satisfies~\eqref{equ:auxprob}, invoking Weyl's decomposition theorem,  
we have that $\vec{q}(\omega) = \A(\omega) \big(\vec{\xi} + \vec{v}_\vec{\xi}(\omega)\big)$ 
belongs to the space $\Lsol$.

By ergodicity of the dynamical system $T$, we can invoke the Birkhoff 
Ergodic Theorem to conclude that, for almost all $\omega \in \Omega$, 
\[
   \vec{p}^\eps_T(\cdot, \omega) \weak \vec{\xi}, \mbox{ in } \L^2(\O),
\quad \mbox{ and } \quad
   \vec{q}^\eps_T(\cdot, \omega) \weak \A^0\vec{\xi}, 
       \mbox { in }  \L^2(\O).
\]

Next, since $\A(\omega) \in \Rnnsym$, we can write, 
\begin{equation} \label{equ:obv}
  \vec{\sigma}^\eps(\vec{x}, \omega) \cdot \vec{p}_T^\eps(\vec{x}, \omega)
 = \A^\eps_T(\vec{x}, \omega)\grad u^\eps(\vec{w}, \omega) \cdot \vec{p}_T^\eps(\vec{x}, \omega)
  = \grad{u}^\eps(\vec{x}, \omega) \cdot \vec{q}^\eps_T(\vec{x}, \omega). 
\end{equation}
Let us consider both sides of \eqref{equ:obv}. Note that
$-\div \vec{\sigma}^\eps(\cdot, \omega) = f$
and $\curl \vec{p}^\eps(\cdot, \omega) = 0$ for every $\eps$;
this along with the weak convergence
of $\vec{\sigma}^\eps(\cdot, \omega)$ and $\vec{p}_T^\eps(\cdot, \omega)$
allows us to use Lemma~\ref{lem:compensated} to get
\begin{equation} \label{equ:weaks-lhs}
\vec{\sigma}^\eps(\cdot, \omega) \cdot \vec{p}_T^\eps(\cdot, \omega)
    \weaks
\vec{\sigma}^0 \cdot \vec{\xi}. 
\end{equation}
On the other hand, considering the right-hand side of \eqref{equ:obv},
we note that for every $\eps$, we have $\curl \grad{u^\eps} = 0$ and
(for almost all $\omega \in \Omega$) $\div{\vec{q}^\eps(\cdot, \omega)} = 0$. 
Therefore, again we use Lemma~\ref{lem:compensated}
to get
\begin{equation} \label{equ:weaks-rhs}
 \grad{u}^\eps(\cdot, \omega) \cdot \vec{q}^\eps_T(\cdot, \omega) 
 \weaks
 \grad{u}^0 \cdot \A^0\vec{\xi}. 
\end{equation}
Finally, using \eqref{equ:obv} along with \eqref{equ:weaks-lhs}
and \eqref{equ:weaks-rhs} we have
$\grad{u^0} \cdot \A^0 \vec{\xi} = \vec{\sigma}^0 \cdot \vec{\xi}$. Therefore, 
by symmetry of $\A^0$ 
\begin{equation*}
   \vec{\sigma}^0 \cdot \vec{\xi} = \A^0 \grad{u^0} \cdot \vec{\xi},
\end{equation*}
and since $\vec{\xi}$ was arbitrary we have $\vec{\sigma}^0 = \A^0 \grad{u^0}$.
Therefore, recalling the definition of $\vec{\sigma}^\eps$ and~\eqref{equ:fluxconv_nd},
we have that
\[
    \A_T^\eps(\cdot, \omega) \grad{u^\eps(\cdot, \omega)} \weak \A^0 \grad{u^0}, \quad \mbox{in } \L^2(\O).
\]
Hence, we can pass to limit $\eps \to 0$ in~\eqref{equ:weakform_nd} to get 
\[
\int_\O \A^0 \grad{u^0} \cdot \grad{\phi} \, d\vec{x} 
= \int_\O f \phi \, d\vec{x}, \quad \forall \phi \in H^1_0(\O),
\]
which says that $u^0$ is weak solution to the problem~\eqref{equ:homogenized_pde}.
Note also that since $\A^0$ and $f$ are deterministic, $u^0$ does not depend on $\omega$.
\end{proof}

\subsection{Variational characterization of the homogenized matrix} 
\label{sec:variational}
Let the probability space $(\Omega, \F, \mu)$ be as in the previous subsection,  
and let $\A \in \Es(\nu_1,\nu_2, \Omega)$ be as in Theorem~\ref{thm:homogthm-main}.
For an arbitrary $\boldsymbol\xi\in\R^n$ we let  
$\J_{\boldsymbol\xi} : \Vpot \to \R$ be the quadratic functional below:
\begin{equation}
\J_{\boldsymbol\xi}(\vec{v}) =
\int_\Omega 
\big(\boldsymbol\xi + \vec{v}(\omega)\big) \cdot \A(\omega) \big(\boldsymbol\xi + \vec{v}(\omega)\big) \, \mu(d\omega), \quad \vec{v} \in \Vpot.
\end{equation}
Note that the dynamical system $T$ in definition of $\Vpot$ here is as in Theorem~\ref{thm:homogthm-main}.
The functional $\J_{\boldsymbol\xi}$ is strictly convex, coercive,
and bounded from below, and therefore, it has a unique minimizer in~$\Vpot$.
The Fr\'{e}chet derivative of $\J_{\boldsymbol\xi}$ at the minimizer
$\vec{v}_{\boldsymbol\xi}$ in any direction $\boldsymbol\varphi$
is zero, that is:
\begin{equation} \label{eq:weakVpot}
        \int_\Omega\A(\omega) \big(\boldsymbol\xi + \vec{v}_{\boldsymbol\xi}(\omega) \big) \cdot
        \boldsymbol\varphi(\omega) \, \mu(d\omega) = 0, \text{ for all } \boldsymbol\varphi \in
        \Vpot.
\end{equation}
Therefore in view of Weyl's decomposition we have
\begin{equation*} 
        \A (\boldsymbol\xi + \vec{v}_{\boldsymbol\xi}) \in \Lsol.
\end{equation*}
It is clear from~\eqref{eq:weakVpot} that
$\vec{v}_{\boldsymbol\xi}$ is linear in~$\boldsymbol\xi$.  Hence,
the expected value
$\ave{\A (\boldsymbol\xi + \vec{v}_{\boldsymbol\xi})}$, viewed as a
function of $\boldsymbol\xi$, is a linear mapping from $\R^n$ to $\R^n$.
Consequently, 
we define the matrix $\A^0$ by
\begin{equation} \label{eq:A0}
        \A^0 \boldsymbol\xi  = \int_\Omega \A(\omega) \big(\boldsymbol\xi +
        \vec{v}_{\boldsymbol\xi}(\omega)\big) \, \mu(d\omega), \quad \vec{\xi} \in \R^n.
\end{equation}
Notice that $\A^0$ defined above is the same as the homogenized matrix
in Theorem~\ref{thm:homogthm-main}. 
\newcommand{\bxi}{\vec{\xi}}
\begin{proposition} \label{lem:A^0}
The homogenized matrix $\A^0$ satisfies the following:
\begin{enumerate}
\item For every $\boldsymbol\xi \in \R^n$, $\displaystyle \boldsymbol\xi \cdot \A^0 \boldsymbol\xi
        = \inf_{\vec{v}\in\Vpot} \J_{\boldsymbol\xi} (\vec{v})$.
\item The matrix $\A^0$ is symmetric and positive definite.
\end{enumerate}
\end{proposition}
\begin{proof} 
Let us note that,
\begin{multline*}
        \inf_{\vec{v}\in\Vpot} \J_{\boldsymbol\xi} (\vec{v})
        = \J_{\boldsymbol\xi}(\vec{v}_{\boldsymbol\xi}) 
        = \int_\Omega \big(\boldsymbol\xi + \vec{v}_{\boldsymbol\xi}(\omega)\big)
        \cdot
        \A(\omega) \big(\boldsymbol\xi + \vec{v}_{\boldsymbol\xi}(\omega)\big) \, \mu(d\omega) \\
        = \boldsymbol\xi \cdot \int_\Omega                    
        \A(\omega) \big(\boldsymbol\xi + \vec{v}_{\boldsymbol\xi}(\omega)\big) \, \mu(d\omega) 
        + \int_\Omega  \vec{v}_{\boldsymbol\xi}(\omega) \cdot 
        \A(\omega) \big(\boldsymbol\xi + \vec{v}_{\boldsymbol\xi}(\omega)) \, \mu(d\omega). 
\end{multline*}
Now, the first integral in the right-hand side reduces to
$ \boldsymbol\xi \cdot \A^0 \boldsymbol\xi $ due to~\eqref{eq:A0}, and 
the second integral vanishes because $\vec{v}_\vec{\xi}$ and  $\A (\vec{\xi} + \vec{v}_{\vec{\xi}})$ 
are orthogonal in $\L^2(\Omega)$.

\newcommand{\ei}{\vec{e}_i}
\newcommand{\ej}{\vec{e}_j}
\newcommand{\vi}{\vec{v}_i}
\newcommand{\vj}{\vec{v}_j}
To show $\A^0$ is symmetric, we proceed as follows. Let $\vec{e}_i$ and $\vec{e}_j$ 
be $i^{th}$ and $j^{th}$ standard basis vectors in $\R^n$, and 
let $\vec{v}_i$ and $\vec{v}_j$
be minimizers in $\Vpot$ of $\J_{\vec{e}_i}$ and $\J_{\vec{e}_j}$ 
respectively. It is straightforward to see $\ei \cdot \A^0 \ej = \int_\Omega (\ei + \vi) \cdot \A (\ej + \vj) \, d\mu$. 
Thus, symmetry of $\A^0$ follows from symmetry of $\A$. 
As for positive definiteness, we note 
\begin{multline*}
   \bxi \cdot \A^0 \bxi 
   = \int_\Omega \big(\bxi + \vec{v}_\bxi(\omega)\big) \cdot \A(\omega)\big(\bxi + \vec{v}_\bxi(\omega)\big) \, \mu(d\omega)
   \geq \nu_1 \int_\Omega |\bxi + \vec{v}_\bxi(\omega)|^2 \, \mu(d\omega)\\ 
   \geq \nu_1 \left| \int_\Omega \big(\bxi + \vec{v}_\bxi(\omega)\big) \, \mu(d\omega) \right|^2
   = \nu_1 |\bxi|^2. \qedhere
\end{multline*}
\end{proof}

\section{Epilogue}\label{sec:epilogue}
In this article we took a brief tour of stochastic homogenization by studying
homogenization of linear elliptic PDEs of divergence form with stationary and
ergodic coefficient functions.  The goal of our discussion was to provide an
accessible entry into a very rich theory that is elaborated in detail in books
such as~\cite{Kozlov94,Chechkin}, which we refer to for in-depth coverage of 
various aspects of stochastic homogenization.  Also,
we mention again the book~\cite{Tartar09} by L.\ Tartar, on the general theory
of homogenization, that is an excellent resource for mathematicians working in
the area as well as those who are entering the field.
We end our discussion by giving some pointers for further reading.

Our discussion focused on homogenization of linear elliptic PDEs with random
coefficients.  The homogenization of \emph{nonlinear} PDEs involves many
additional difficulties both in theory as well as in numerical computations.
We refer to the book~\cite{Pankov97} as well as the
articles~\cite{dal1986a,dal1986b,caffarelli2005homogenization,caffarelli2010rates}
for stochastic homogenization theory for nonlinear problems.  See
also~\cite{efendiev-pankov-04b,efendiev-pankov-04a}, which concern numerical
methods for stochastic homogenization of nonlinear PDEs.

Stochastic homogenization continues to be an active area of research.  Recent
developments in the area
include the works~\cite{blanc2012variance,gloria2012numerical,gloria2013quantification,
gloria2013quantitative,gloria2011optimal,ConlonSpencer2014}. We also point to
the survey article~\cite{gloria2012survey}, which provides a review of the
state-of-the-art of numerical methods for homogenization of linear elliptic
equations with random coefficients. Recent works in homogenization of random
nonlinear PDEs include the articles~\cite{armstrong2012,armstrong2014}.  See
also~\cite{ArmstrongSouganidis13,ArmstrongCardaliaguetSouganidis14}, which
concern stochastic homogenization of Hamilton-Jacobi equations.

\section*{Acknowledgements}
I would like to thank Matteo Icardi and H{\aa}kon Hoel for reading through an earlier draft
of this work and giving me helpful feedback.
\bibliographystyle{abbrv}
\bibliography{refs}
\end{document}